\documentclass[11pt,oneside]{amsart}
\usepackage{amsmath}
\usepackage{amssymb}
\usepackage{eucal}
\usepackage{amsthm}
\usepackage{amscd}
\usepackage{hyperref}
\usepackage{soul}
\usepackage{url}
\usepackage{skull}
\usepackage{adjustbox}
\usepackage{mathrsfs}

\usepackage{tikz}
\usetikzlibrary{shapes.geometric, arrows}
\usetikzlibrary{shapes,arrows}
\usepackage[latin1]{inputenc}
\usepackage{tikz}
\usetikzlibrary{shapes,arrows}
\usetikzlibrary{shapes.multipart}
\usepackage{verbatim}

\usepackage{todonotes}

\usepackage{amssymb}
\usepackage[]{amsmath, amsthm, amsfonts,graphicx, amscd,}
\usepackage[all,cmtip]{xy}
\input amssym.def \input amssym
\usepackage{forest}
\usepackage{mdframed}
\usepackage{framed}

\begin{document}

\newtheorem{thm}{Theorem}[section]
\newtheorem{lthm}{Lost Theorem}
\renewcommand{\thelthm}{of Chazy \& Painlev\'e}
\newtheorem*{srp}{The Shelah reflection principle}
\newtheorem{claim}[thm]{Claim}
\newtheorem {fact}[thm]{Fact}
\newtheorem{con}[thm]{Conjecture}
\newtheorem*{corrB}{Corollary B$^*$}

\newtheorem*{thmstar}{Theorem}
\newtheorem*{prob}{Problem}
\newtheorem{prop}[thm]{Proposition}
\newtheorem{cor}[thm]{Corollary}
\newtheorem*{propstar}{Proposition}
\newtheorem {lem}[thm]{Lemma}
\newtheorem*{lemstar}{Lemma}
\newtheorem{conj}[thm]{Conjecture}
\newtheorem{question}[thm]{Question}
\newtheorem*{questar}{Question}
\newtheorem{ques}[thm]{Question}
\newtheorem*{conjstar}{Conjecture}
\newtheorem{fct}[thm]{Fact}
\theoremstyle{remark}
\newtheorem{rem}[thm]{Remark}
\newtheorem{exmp}[thm]{Example}
\newtheorem{cond}[thm]{Condition}
\newtheorem{np*}{Non-Proof}
\newtheorem*{remstar}{Remark}
\theoremstyle{definition}
\newtheorem{defn}[thm]{Definition}
\newtheorem*{defnstar}{Definition}
\newtheorem{exam}[thm]{Example}
\newtheorem*{examstar}{Example}
\newtheorem{assump}[thm]{Assumption}
\newtheorem{Thm}[thm]{Theorem}

\theoremstyle{plain}
\newtheorem{thmx}{Theorem}
\renewcommand{\thethmx}{\Alph{thmx}}
\newtheorem{corx}[thmx]{Corollary}

\newtheorem{innerimportantthm}{Theorem}
\newenvironment{importantthm}[1]
  {\renewcommand\theinnerimportantthm{#1}\innerimportantthm}
  {\endinnerimportantthm}
  
\newtheorem{innerimportantcor}{Corollary}
\newenvironment{importantcor}[1]
  {\renewcommand\theinnerimportantcor{#1}\innerimportantcor}
  {\endinnerimportantcor}

\def \ta {\tau_{\mathcal{D}/\Delta}}
\def \D {\Delta}
\def \DD {\mathcal D}

\newcommand{\gen}[1]{\left\langle#1\right\rangle}

\newcommand{\pd}[2]{\frac{\partial #1}{\partial #2}}

\newcommand{\codim}{\text{codim}\,}
\newcommand{\td}{\text{tr.deg.}}
\newcommand{\pp}{\partial }
\newcommand{\pdtwo}[2]{\frac{\partial^2 #1}{\partial #2^2}}
\newcommand{\od}[2]{\frac{d #1}{d #2}}
\def\Ind{\setbox0=\hbox{$x$}\kern\wd0\hbox to 0pt{\hss$\mid$\hss} \lower.9\ht0\hbox to 0pt{\hss$\smile$\hss}\kern\wd0}
\def\Notind{\setbox0=\hbox{$x$}\kern\wd0\hbox to 0pt{\mathchardef \nn=12854\hss$\nn$\kern1.4\wd0\hss}\hbox to 0pt{\hss$\mid$\hss}\lower.9\ht0 \hbox to 0pt{\hss$\smile$\hss}\kern\wd0}
\def\ind{\mathop{\mathpalette\Ind{}}}
\def\nind{\mathop{\mathpalette\Notind{}}}
\numberwithin{equation}{section}

\def\id{\operatorname{id}}
\def\Frac{\operatorname{Frac}}
\def\Const{\operatorname{Const}}
\def\spec{\operatorname{Spec}}
\def\span{\operatorname{span}}
\def\exc{\operatorname{Exc}}
\def\Div{\operatorname{Div}}
\def\cl{\operatorname{cl}}
\def\mer{\operatorname{mer}}
\def\trdeg{\operatorname{trdeg}}
\def\ord{\operatorname{ord}}

\newcommand{\m}{\mathbb }
\newcommand{\mf}{\mathfrak }
\newcommand{\is}{^{p^ {-\infty}}}
\newcommand{\QQ}{\mathbb Q}
\newcommand{\fh}{\mathfrak h}
\newcommand{\CC}{\mathbb C}
\newcommand{\RR}{\mathbb R}
\newcommand{\ZZ}{\mathbb Z}
\newcommand{\tp}{\operatorname{tp}}
\newcommand{\SL}{\operatorname{SL}}

\title[A differential approach to Ax-Schanuel, I]{A differential approach to Ax-Schanuel, I}

\author[D. Bl\'azquez-Sanz]{David Bl\'azquez-Sanz}
\address{David Bl\'azquez-Sanz, Universidad Nacional de Colombia - Sede Medell\'in, Facultad de Ciencias, Departamento de Matem\'aticas, Colombia}
\email{ dblazquezsa@unal.edu.co}

\author[G. Casale]{Guy Casale}
\address{Guy Casale, Univ Rennes, CNRS, IRMAR-UMR 6625, F-35000 Rennes, France}
\email{guy.casale@univ-rennes1.fr}

\author[J. Freitag]{James Freitag}
\address{James Freitag, University of Illinois Chicago, Department of Mathematics, Statistics,
and Computer Science, 851 S. Morgan Street, Chicago, IL, USA, 60607-7045.}
\email{jfreitag@uic.edu}

\author[J. Nagloo]{Joel Nagloo}
\address{Joel Nagloo, University of Illinois Chicago, Department of Mathematics, Statistics,
and Computer Science, 851 S. Morgan Street, Chicago, IL, USA, 60607-7045.}
\email{jnagloo@uic.edu}

\thanks{G. Casale is partially funded by CAPES-COFECUB project MA 932/19 and Centre Henri Lebesgue, program ANR-11-LABX-0020-0. J. Freitag is partially supported by NSF CAREER award 1945251. J. Nagloo is partially supported by NSF grant DMS-2203508 and DMS-2348885. D. Bl\'azquez-Sanz is supported by UNAL project Hermes-60920 during the final stages of this research. Some of the work in this paper was completed while J. Nagloo was supported by an AMS Centennial Fellowship.}
\date{\today}
\maketitle

\maketitle

\begin{abstract} In this paper, we prove several Ax-Schanuel type results for uniformizers of geometric structures; our general results describe the differential algebraic relations between the solutions of the partial differential equations satisfied by the uniformizers. In particular, we give a proof of the full Ax-Schanuel Theorem with derivatives for some uniformizers of simple projective structure on curves including unifomizers of any Fuchsian group of the first kind and any genus.

Combining our techniques with those of Ax, we give a strong Ax-Schanuel result for the combination of the derivatives of the $j$-function and the exponential function. In the general setting of Shimura varieties, we obtain an Ax-Schanuel theorem for the derivatives of uniformizing maps.

Our techniques combine tools from differential geometry, differential algebra and the model theory of differentially closed fields. 
\end{abstract}

\section{Introduction}

In this paper we use techniques from differential geometry, differential algebra and the model theory of differentially closed fields to prove several functional transcendence results. Our main result is stated in the context of a $G$-principal bundle $\pi:P \to Y$ with $Y$ a {irreducible} complex algebraic variety and $G$ a complex algebraic group. We explain the notation of our main theorem after its statement: 


\begin{thmx}\label{thmA}
Let $\nabla$ be a  (possibly irregular) $G$-principal flat connection on $P \to Y$ with sparse Galois group ${\rm Gal}(\nabla)=G$.  Let $V$ be an irreducible algebraic subvariety of $P$, $\mathcal{L}\subset P$  a $\nabla$-horizontal leaf and $\mathcal{V} \subset V\cap \mathcal L$ an irreducible analytic subvariety of $P$. If $\dim V < \dim \mathcal V + \dim G$ 
then the projection of $\mathcal V$ in $Y$ is contained in a $\nabla$-special subvariety.
\end{thmx} 

By a {\it $\nabla$-special subvariety} $X \subset Y$, we mean an irreducible maximal subvariety such that ${\rm Gal}(\nabla|_X)$ is a strict subgroup of $G={\rm Gal}(\nabla)$. 
 
The hypothesis  ${\rm Gal}(\nabla)=G$ means that the horizontal leaves are Zariski dense in $P$. Hence $\nabla$-special subvarieties $X \subset Y$ are those (maximal) subvarieties where the horizontal leaves of the restriction $\nabla_{|X}$ are not Zariski dense in $P_{|X}$. {\it Sparse groups} are introduced and defined on page \pageref{def:NZDS} in Definition \ref{def:NZDS}. Roughly speaking, these are groups for which the proper analytic subgroups are never Zariski dense. For instance, semi-simple groups are sparse, while algebraic tori are not.

\begin{rem}
As $\mathcal L \subset P$ is an immersed analytic submanifold (but not an analytic subvariety) of codimension $q = \dim G$, it might intersect algebraic subvarieties along a codimension $q$ immersed analytic submanifold. The condition in Theorem \ref{thmA} means that $V$ is atypical with respect to $\nabla$. 
\end{rem}

Theorem \ref{thmA} can be used to prove most of the Ax-Schanuel type resuts recently obtained in the context of variations of mixed $\mathbb Z$-Hodge structure.

In this context, the connection $\nabla$ is singular-regular and using a theorem of Andr\'e-Deligne, Chiu proved in \cite[Theorem 3.2]{Chiu}, that the $\nabla$-special subvarieties are then \emph{weakly special} ones. Then from Theorem \ref{thmA} one can recover the main result of Bakker and Tsimermann \cite[Theorem 1.1]{BakkerTsimermann2} on the Ax-Schanuel conjecture for variations of Hodge structures. Similarly, using Theorem A, one can recover the main result of Gao and Klinger \cite[Theorem 1.1]{GaoKlinger}, and Chiu \cite[Theorem 1.2]{Chiu0} on the Ax-Schanuel conjecture for variations of mixed Hodge structures. In fact, Chiu derives the more general variants of these results including derivatives \cite[Theorem 1.2 and Corollary 1.4]{Chiu} using Theorem \ref{thmA} together with the results of \cite{BBT}.

Theorem \ref{thmA} also implies many recently proved functional transcendence results for uniformizing functions associated with discrete groups, many fitting into a generalization of the setting of \cite{Scanlon}. To that end, we give a consequence of Theorem A suited to covering maps by defining the notion of a $(G, G/B)$-structure on an algebraic variety $Y$ in Section 3.2. 



For instance, these geometric structures includes  algebraic varieties that are quotients of a bounded symmetric domain $\Omega$ by a lattice $\Gamma\subset G(\m R)$ in its group of holomorphic automorphisms. It also includes the differential equations satisfied by conformal mappings of circular polygons \cite[Chapter 4]{circpoly}. In any case, we will show that attached to any $(G, G/B)$-structure is a $G$-principal connection $\nabla$ and so are able to apply Theorem \ref{thmA} and obtain the following result.

\begin{corx}\label{corB}
Let $\upsilon$ be the uniformization of an irreducible $(G,G/B)$-structure on an algebraic variety $Y$ with $G$ sparse.
Assume $W \subset G/B \times Y$ is an irreducible algebraic subvariety intersecting the graph of $\upsilon$. Let $U$ be an irreducible complex analytic component of this intersection such that
$$
\dim W < \dim U + \dim Y.
$$
Then the projection of $U$ to $Y$ is contained in a $\nabla$-special subvariety of $Y$.
\end{corx}

The classical and most studied examples of such structures come from Shimura varieties which we describe next, but our results apply to more general situations. Take $G$ to be a connected semi-simple algebraic $\mathbb{Q}$-group. Take $K$ to be a maximal compact subgroup of $G(\m R)$ and suppose that  $\Omega=G(\mathbb{R})/K$ is a bounded symmetric domain (this is quite rare). Then the compact dual $\check\Omega$ of $\Omega$ is given as the quotient $\check\Omega=G(\mathbb{C})/ \mathcal P$ for a parabolic subgroup $\mathcal P$ of $G$. This quotient $\check\Omega$ is a homogeneous projective variety and $\Omega$ is a semi-algebraic subset (if we assume $K\subset \mathcal P$). Given an arithmetic lattice $\Gamma\subset G(\m Q)$, the analytic quotient $Y:=\Gamma \backslash \Omega=\Gamma \backslash G(\mathbb{R})/K$ has the structure of an algebraic variety. The above data defines the pure connected Shimura variety, $Y$. As detailed in Subsection \ref{sectionUniformization}, a $(G, G/B)$-structure on $Y$ where $B=\mathcal P$ can be taken to be the system of partial differential equations satisfied by a uniformization function $\upsilon:\Omega\rightarrow Y :=\Gamma \backslash \Omega$. This system or its solution set is denoted by $\mathscr Y$.

There, the Ax-Schanuel theorem follows from Corollary \ref{corB} and allows one to recover the main result of Mok, Pila and Tsimerman \cite[Theorem 1.1]{ASmpt}\footnote{Additional consequences of the Corollary \ref{corB} are detailed in Section 4.}:

\begin{corrB}\label{corB1}
Let $v: \Omega \rightarrow Y = \Gamma \backslash \Omega$ be the uniformization function of a Shimura variety $Y= \Gamma \backslash G( \m R) /K$. Assume $W \subset G/B \times Y$ is an irreducible algebraic subvariety intersecting the graph of $v$. Let $U$ be an irreducible complex analytic component of this intersection such that
$$
\dim W < \dim U + \dim Y.
$$
Then the projection of $U$ to $Y$ is contained in a weakly special subvariety of $Y$.
\end{corrB}

If we assume in addition that the $(G, G/B)$-structure is {\em simple} (see Definition \ref{simple}), a natural assumption, then we are also able to obtain an Ax-Schanuel Theorem (including with derivatives) for products of $Y$. This gives a result which applies to more general situations than existing Ax-Schanuel type results (e.g. the main theorem of \cite{ASmpt}):

\begin{thmx}\label{thmC}
Let $(Y, \mathscr Y)$ be a simple $(G,G/B)$-structure on $Y$, $\hat p_1, \ldots , \hat p_n$ be $n$ formal parameterizations of (formal) neighborhoods of points $p_1, \ldots, p_n$ in $G/B$ and $\upsilon_1, \ldots , \upsilon_n$ be solutions of $\mathscr Y$ defined in a neighborhood of $p_1, \ldots, p_n$ respectively.
If $$
{\rm tr.deg.}_{\mathbb C} \mathbb C\left( \hat p_i , (\partial^\alpha \upsilon_i)(\hat p_i) : 1 \leq i \leq n,\; \alpha \in \mathbb N^{\dim G/B} \right) < \dim Y + n\dim G$$
then there exist $i<j$ such that 
$$
{\rm tr.deg.}_{\mathbb C} \mathbb C (\upsilon_i(\hat p_i), \upsilon_j(\hat p_j)) = {\rm tr.deg.}_{\mathbb C} \mathbb C(\upsilon_i(\hat p_i) ) = {\rm tr.deg.}_{\mathbb C} \mathbb C(\upsilon_j(\hat p_j)) = \dim G/B.$$
\end{thmx} 

By construction $
{\rm tr.deg.}_{\mathbb C} \mathbb C\left( \hat p_i , (\partial^\alpha \upsilon_i)(\hat p_i) : 1 \leq i \leq n,\; \alpha \in \mathbb N^{\dim G/B} \right) \leq n(\dim Y + \dim G)$ so the hypothesis imposes $(n-1)\dim Y +1$ additional algebraic relations. Theorem \ref{thmC} does not (in full generality) give any details about the kinds of special subvarieties or correspondences that can occur - this is a problem we address in several concrete situations of interest in this paper and in additional cases in the sequel. Nevertheless, it will be crucial in giving a model theoretic analysis of the relevant partial differential equations solving several cases of a problem posed by Scanlon \cite{Scanlon}. For example, as a consequence we are able to show that the sets $\mathscr{Y}$ defined by the equations (along with some natural inequalities), in a differentially closed field, are strongly minimal and geometrically trivial. 

Building on this model theoretic analysis, we are also able to give a complete analysis of algebraic relations in the case of hyperbolic curves. Let $\Gamma\subset {\rm PSL}_2(\mathbb{R})$ be a Fuchsian group of the first kind and let $j_{\Gamma}$ be a uniformizing function for $\Gamma$. Notice here that $G={\rm PSL}_2(\mathbb C)$, $\Omega=\mathbb{H}$ and $\check\Omega=\mathbb{CP}_1$. Let $\hat p_1,\ldots,\hat p_n$ be formal parameterizations of neighborhoods of points $p_1, \ldots, p_n$ in $\m H$. We write $\hat p_i(s_1, \ldots, s_\ell)$ for these non constant formal power series in $\ell$ variables. We prove the Ax-Schanuel Theorem with derivatives for $j_{\Gamma}$:

\begin{thmx}\label{thmd}
Assume that $\hat p_1,\ldots,\hat p_n$ are geodesically independent, namely $\hat p_i$ is nonconstant for $i =1, \ldots , n$ and there are no relations of the form $\hat p_i = \gamma \hat p_j$ for $i \neq j$, $i,j \in \{1, \ldots , n \}$ and $\gamma$ is an element of  ${\rm Comm}(\Gamma)$, the commensurator of $\Gamma$. Then
$${\rm tr.deg.}_{\mathbb{C}}\mathbb{C}(\hat p_1,j_{\Gamma}(\hat p_1),j'_{\Gamma}(\hat p_1),j''_{\Gamma}(\hat p_1),\ldots,\hat p_n,j_{\Gamma}(\hat p_n),j'_{\Gamma}(\hat p_n),j''_{\Gamma}(\hat p_n))\geq 3n + {\rm rank}\left(\frac{\partial \hat p_j}{\partial s_i}\right).$$
\end{thmx}

Notice that the statement given above, in the spirit of Ax's original paper, is both more general and stronger than the one found in \cite{AXS} whereby $\hat p_1,\ldots,\hat p_n$ are assumed to be the coordinate functions of some complex {\em analytic} subvariety of some open subset $D\subset \m H^n$.  For instance, Theorem \ref{thmd} strengthens the Ax-Lindemann-Weierstrass theorem of \cite{CasFreNag}, and also generalizes the setup by dropping the assumption that the quotient is genus zero. The technique of the proof is also completely different. 

Finally, using Theorem \ref{thmd} and the techniques developed in the paper, we obtain an Ax-Schanuel Theorem with derivatives for the exponential function together with any uniformizing function $j_{\Gamma}$ of some Fuchsian group $\Gamma\subset {\rm PSL}_2(\mathbb{R})$ of the first kind. 

\begin{thmx}\label{ThmE}
Let $\hat p_1,\ldots,\hat p_{k+\ell}$ be formal parameterizations (in variables $s_1, \ldots, s_m$ as above) of neighborhoods of points $p_1, \ldots, p_{k+\ell}$ in $\m H$. If 
$${\rm tr.deg.}_{\mathbb C}\mathbb C \left(\hat p_1, \ldots \hat p_{k+\ell}, \exp(\hat p_1), \ldots\exp(\hat p_k), j_{\Gamma} (\hat p_{k+1}), \ldots j_{\Gamma} ''(\hat p_{k+\ell})\right) < k + 3\ell+{\rm rank}\left(\frac{\partial \hat p_j}{\partial s_i}\right)$$ 
then 
\begin{itemize}
\item there exists $n\in \mathbb Z^k\setminus \{0\}$ such that $\sum_{i=1}^k n_i \hat p_i \in \mathbb C$; or
\item there exist $k+1\leq i < j\leq k+\ell$ and $\gamma \in {\rm Comm}(\Gamma)$ such that $\hat p_i = \gamma\hat p_j$.
\end{itemize}
\end{thmx}
In forthcoming work, van Hille, Jones, Kirby, and Speissegger use Theorem \ref{ThmE} to show that there is a Pfaffian chain $f_1, \ldots f_{\ell} : (0,1) \rightarrow \m R$ such that the theory of the first order structure $(\m R , + , \cdot , f_1, \ldots, f_{\ell} )$ is not model complete. The question of the existence of such a chain was an old open question in the area of o-minimality. This is part of a growing connection between differential algebraic properties of analytic functions and o-minimality (cf. \cite{notpfaff} and also \cite{notpfaff1} where Theorem \ref{ThmE} is used to show that the $j$-function is not Pfaffian over the real exponential field). 

\subsection{Applications of Ax-Schanuel theorems} 
Over the last decade, functional transcendence results, often in the form of the Ax-Lindemann-Weierstrass theorem type results for certain analytic functions have played a key role in the class of diophantine problems known as \emph{special points} problems or problems of \emph{unlikely intersections}. See for instance, \cite{Pila2, PTsurface, TsAOAg, fresca, CasFreNag}. The Ax-Schanuel theorem generalizes the Ax-Lindemann-Weierstrass. In the setting of pure Shimura varieties, the Ax-Schanuel theorem of \cite{ASmpt} has recently been applied to certain diophantine problems \cite{daw2018applications}. Various Ax-Schanuel results have also been applied to cases of the Zilber-Pink conjecture \cite{AslanyanZP, ZPabel}. 

Over the past several years, in a series of works Aslanyan, and later Aslanyan, Kirby and Eterovi\'c \cite{Aslanyan0, Aslanyan, Aslanyan1, Aslanyan2, Aslanyan3} develop the connection between Ax-Schanuel transcendence statements and the existential closedness of certain reducts of differentially closed fields related to equations satisfied by the $j$-function. This series of work builds on the earlier program of Kirby, Zilber and others mainly around the exponential function, see e.g. \cite{KZ}. We expect our results to contribute nontrivially to this line of work. The earlier work on the exponential function utilizes the transcendence results of Ax \cite{Ax71}, where a differential algebraic proof of the functional version of Schanuel's conjecture is given. Intermediate differential algebraic results from Ax's work are utilized in the program studying existential closedness results around the exponential function. A different approach was required in \cite{Aslanyan0, Aslanyan, Aslanyan1, Aslanyan2, Aslanyan3} for studying the $j$-function, in part since the intermediate results of the functional transcendence results of \cite{AXS} take place in the o-minimal rather than abstract differential setting. For instance, the motivation for a differential algebraic proof of Ax-Schanuel results and this issue is pointed out specifically following Theorem 1.3 of \cite{Aslanyan1} and in Section 4.4 of \cite{Aslanyan3}. Our results open up the possibility of adapting a similar approach to existential closedness around more general automorphic functions since the general technique of our proof is differential algebraic along lines similar to Ax's work. Besides this issue, it is expected that our generalizations of the Ax-Schanuel results of \cite{AXS} can be used to establish existential closedness results for more general automorphic functions beyond the modular $j$-function. 

\subsubsection*{Work of Dogra}
After the initial appearance of our preprint, Dogra \cite{Dogra} proved diophantine results regarding the set of \emph{low rank} points in fibre powers of families of curves. His work uses Theorem \ref{thmA} (along with some refinements) extensively (e.g. \cite[Lemma 11, Lemma 12, Theorem 2, Proposition 6, Proposition 9, Corollary 4, Corollary 5]{Dogra}). Dogra's applications are a specific realization of part of our motivation for pursuing Ax-Schanuel type functional transcendence results in a general differential-geometric/differential-algebraic setting. 

\subsubsection*{Work of Bakker and Tsimerman}
Recently, Bakker and Tsimerman have given a proof of a geometric version of Andr\'e's generalization of the Grothendieck period conjecture \cite{BakkerTsimermann1}. Their Theorem 1.1 is stated in different language - using Nori motives, but they explain how it is equivalent to an Ax-Schanuel type statement \cite[Theorem 1.3]{BakkerTsimermann1}. They give a short proof of their Ax-Schanuel result using Theorem \ref{thmA}.

\subsubsection*{Work of Pila and Scanlon} Recently, Pila and Scanlon \cite{PilaScanlon} proved function field versions of the Zilber-Pink conjectures for varieties supporting a variation of Hodge structures. Using these results, they show that the differential equations associated with Shimura varieties give examples of minimal (sometimes strongly minimal) types with trivial forking geometry extending some of the results in Section 4.

\subsection{Organization of the paper} The paper is organized as follows. In Section 2 we give the necessary background in the Cartan approach to the study of linear differential equations. In particular, we recall the definition (and basic properties) of a $G$-principal flat connection and its associated connection form. In Section 3 we prove Theorem \ref{thmA} and derive some of its corollaries. We also introduce the idea of geometric structures (or $(G,G/B)$ structures) and apply Theorem \ref{thmA} in this setting to obtain Corollary \ref{corB} . In Section 4 we recall Scanlon's work on covering maps, show that they are part of the formalism of geometric structures and detail the intersection of our work with other similar work in the literature. In section 5 we use Theorem \ref{thmA} to study products of geometric structures (Theorem \ref{thmC}) and use those to give a model theoretic study of the $(G,G/B)$ structures. Section 6 and 7 are devoted to proving Theorem \ref{thmd}, i.e., the full Ax-Schanuel Theorem with derivatives in the case of hyperbolic curves. Finally, we prove Theorem \ref{ThmE}, the Ax-Schanuel theorem with derivatives for the combined case of the exponential function and any uniformizer of an hyperbolic curve, in Section 8.

\subsection*{Acknowledgements}
The main ideas behind this work were developed at the School of Mathematics at the Institute for Advanced Study as a part of the 2019 Summer Collaborators Program. We thank the Institute for its generous support and for providing an excellent working environment. The authors also thank Peter Sarnak for useful conversations during our time at the IAS. During the preparation of this manuscript, the authors were also supported by the American Institute of Mathematics, through their SQUARE program. We thank AIM for the productive atmosphere for pursuing mathematics.  We also thank the anonymous referee for important comments and suggestions. Finally, we would like to express our gratitude to B\'elinda C.Z., Kate F. and Sharonne C. for their support, understanding, and infinite patience during the work and preparation of this article.

\section{The Cartan approach to linear differential equations}\label{cartan}

In this section, we set up some notation and conventions about principal connections.
Algebraic principal bundles were defined in \cite{serre} and most of the usual construction of differential geometry can be adapted to the algebraic framework under suitable hypotheses, see \cite[appendix A]{bost}. The notion of connection form of a principal connection  does not seem to have been used in the algebraic context. 

In the $\mathscr C^{\infty}$ category, a complete reference for the definition and the properties of the connection form is Kobayashi-Nomizu \cite{kobayashi} or Sharpe's book \cite{Sh}. Throughout, we will be working over the field of complex numbers $\mathbb C$. Analytic functions or manifolds mean holomorphic.  

\subsection{Principal connection} 

Let $G$ be an algebraic group, $Y$ a smooth irreducible algebraic variety and $\pi\colon P\to Y$ a principal bundle modeled over $G$, {\it i.e.} endowed with a (right) action of $G$, denoted by $R$ or by $\cdot$, that induces an isomorphism
$$P \times G \overset{\sim}{\to} P \underset{Y}{\times} P, \quad 
(p,g) \mapsto (p,p\cdot g) = (p,R_g(p)).$$
The fibers of $\pi$ are principal homogeneous $G$-spaces. The election of a point $p$ in a fiber $P_y = \pi^{-1}(y)$ induces an isomorphism of $G$-spaces,
$$
G\xrightarrow{\sim} P_y, \quad g\mapsto p\cdot g,
$$
and an isomorphism of groups,
$$
G\xrightarrow{\sim} {\rm Aut}_G(P_y),\quad g\mapsto \sigma \text{ with } \sigma(p\cdot h) = p\cdot gh
$$ 
between $G$ and the group  ${\rm Aut}_G(P_y)$ of $G$-equivariant automorphisms of $P_y$. Note that this pair of isomorphisms conjugate the left action of $G$ on itself with the action of ${\rm Aut}_G(P_y)$ on $P_y$. 
A \emph{gauge transformation} of $P$ is a $G$-equivariant map $F\colon P \to P$ such that $\pi \circ F = \pi$. That means that for each fiber $F|_{P_y}\in {\rm Aut}_G(P_y)$. \\

We define the \emph{vertical bundle} $T(P/Y)$ as the kernel of $d\pi$, it is a subbundle of $TP$. An algebraic \emph{connection} is a section $\nabla$ of the exact sequence,
$$
0 \to T(P/Y) \to TP \to TY \underset{Y}{\times} P \to 0.
$$
Thus, it is a map
$$
\nabla \colon TY \underset{Y}{\times} P  \to TP ,\quad (v,p) \mapsto \nabla_{v,p},
$$
satisfying $d\pi(\nabla_{v,p}) = v$. 

The image $\nabla(TY\times_Y P)\subset TP$ is a \emph{distribution of vector fields
\footnote{Let $D \subset TP$ be a subset of the tangent bundle of a smooth manifold $P$ such that for each $p \in P$, the fiber above $p$, $D_p \subseteq T_p P$ is a $d$-dimensional subspace. Further, suppose that for any point $p$ there is an Zariski open neighborhood $U$ of $p$ such that for any $z \in U$, we have independent regular vector fields $X_1 (z), \ldots , X_d (z)$ whose span is $D_z$.} 
of rank $\dim Y$ on $P$}
, the so-called $\nabla$-horizontal distribution $\mathcal H_{\nabla}$. We have a canonical decomposition of the tangent bundle $TP = T(P/Y)\oplus {\mathcal H}_\nabla$ as the direct sum of its vertical and $\nabla$-horizontal subbundles. \\

A connection induces a $\nabla$-horizontal lift operator of vector fields on $Y$ to $P$,
$$\nabla \colon \mathfrak X_Y \to \mathfrak \pi_*\mathfrak X_P ; \;\;\;v \mapsto \nabla_v \text{ with } (\nabla_v)(p) = \nabla_{v(\pi(p)),p}$$
that lifts vector fields in an open subset $ U\subset Y$ up to $\nabla$-horizontal vector fields in $\pi^{-1}(U)\subset P$.

This operator is $\mathcal O_X$-linear. By abuse of notation, this lift operator is represented by the same symbol $\nabla$. It completely determines the connection. Note that usually, the symbol $\nabla$ is used to denote the associated covariant derivative. The horizontal lift of rational vector fields on $Y$ defines a map ${\rm D} : \mathbb C(P) \to \mathbb C(P)\otimes_{\mathbb C(Y)} \Omega^1(Y)$ where $\Omega^1(Y)$ is the $\mathbb C(Y)$-vector space of rational $1$-forms. Such a ${\rm D}$ extends the differential structure of $\mathbb C(Y)$ given by the exterior derivative, ${\rm d}: \mathbb C(Y) \to \Omega^1(Y)$, and satisfies Leibniz rule :  $ {\rm D}(ab) = a{\rm D}(b) + b{\rm D}(a)$.\\

We say that the connection $\nabla$ is principal if it is $G$-equivariant: for all $(v,p) \in TY \times_Y P$  and $g \in G$ we have $\nabla_{v,R_g(p)} = dR_g(\nabla_{v,p})$. This is equivalent to requiring that the $\nabla$-horizontal distribution is $G$-invariant, or that the image of the $\nabla$-horizontal lift operator consists of $G$-invariant vector fields. We say that the connection $\nabla$ is \emph{flat} if the lift operator is a Lie algebra morphism, that is:
$$\nabla_{[v,w]} = [\nabla_v,\nabla_w].$$
 
{\bf In what follows, a connection means a $G$-principal flat connection.}\\ 

We are mostly concerned with \emph{rational connections}. The definition above is the definition of an algebraic connection. Let $P\to Y$ be a dominant map of algebraic varieties. A rational principal connection on $P$ over $Y$ is a principal connection on $P|_{Y^o}$ with $Y^{o}$ a Zariski open and dense subset $Y$ such that $P|_{Y^o}$ is a principal bundle.

\begin{exam}\label{ex:linear}
The most important example of a rational $G$-invariant connection is a linear differential equation in fundamental form. Let us fix an algebraic irreducible curve $Y$ and a non constant rational function $y\in \mathbb C(Y)$. As $\mathbb C(Y)$ is an algebraic extension of $\mathbb C(y)$, the derivation $\frac{d}{dy}$ can be uniquely extended to $\mathbb C(Y)$. Our equation is  

$$
\frac{dU}{dy} = A(y)U \text{ with  $U$ an invertible $n\times n$ matrix of unknowns and } A \in \mathfrak{gl}_n(\mathbb C(Y)).
$$
In this situation $G$ is the linear group ${\rm GL}_n(\mathbb C)$ with coordinate ring $\mathbb C\left[ U_i^j , \frac{1}{det(U)}\right]$, $P$ is $Y \times G$, and the action of $G$ on $P$ is given by right translations: $(y,U)\cdot g = (y,U g)$. We see $\frac{d}{dy}$ as a rational vector field on $Y$ and the linear differential system gives us its $\nabla$-horizontal lift 

$$\nabla _{\frac{d}{dy}} = \frac{\partial }{\partial y} + \sum_{i,j,k} A_i^j U_j^k\frac{\partial }{\partial U_i^k}. $$

The $\nabla$-horizontal lift operator is determined by the above formula and $\mathcal O_Y$-linearity. The field $\mathbb{C}(P)$ with the derivation $\nabla_{\frac{d}{dy}}$ is called the field of the universal solution of the linear equation. 
As the vector fields $\sum_k U_j^k\frac{\partial}{\partial U_i^k}$ are right invariant on $G$, $\nabla$ is a principal connection. The compatibility with the Lie bracket is straightforward as for any couple of vector fields $v$ and $w$ on $Y$, $[v,w]$ is colinear to $v$; thus $\nabla$ is flat. The equation defines an algebraic connection outside of the set of poles of the rational functions $A_i^j$ and zeroes of $\frac{d}{dy}$ as a vector field on $Y$.
\end{exam}


\subsection{Basic facts about singular foliations}
We will use freely vocabulary and results from holomorphic foliation theory while recalling the relevant objects in this subsection.

A \emph{singular foliation} $\mathcal F$ of rank $m$ on an algebraic variety $P$ is an $m$-dimensional $\mathbb C(P)$-vector subspace of rational vector fields in $P$, $\mathfrak X_P$, stable by Lie bracket. Such vector fields are said to be tangent to $\mathcal F$. We say that the foliation $\mathcal F$ is regular at $p\in M$ if there is a basis $\{v_1,\ldots,v_m\}$ of $\mathcal F$ such that $v_1(p),\ldots,v_m(p)$ are well defined and $\mathbb C$-linearly independent, otherwise we say that $\mathcal F$ is singular at $p$. The set of singular points of $\mathcal F$ form a Zariski closed subset ${\rm sing}(\mathcal F)$ of codimension $\geq 2$. We say that the foliation $\mathcal F$ is regular if ${\rm sing}(\mathcal F) = \emptyset$.

An \emph{integral submanifold} of $\mathcal F$ is an $m$-dimensional immersed analytic submanifold $\mathcal S\subset P$ (not necessarily embedded in $P$) whose tangent space at each point is generated by the values of vector fields in $\mathcal F$. Maximal connected integral submanifolds are called \emph{leaves}. Frobenius' theorem ensures that through any regular point  passes a unique leaf. Any connected integral submanifold of $\mathcal F$ determines completely a leaf by analytic continuation. 
For general results about Zariski closures of leaves of singular foliations we refer to \cite{bonnet}, in particular the Zariski closure of a leaf is irreducible. 

A subvariety $V\subset P$ is $\mathcal F$-invariant if the vector fields tangent to $\mathcal F$ whose domain is dense in $V$ restrict to rational vector fields on $V$. In such case $\mathcal F|_V$ is a singular foliation on $V$ of the same rank. Leaves of $\mathcal F|_V$ are leaves of $\mathcal F$ contained in $V$.

We say that $\mathcal F$ is \emph{irreducible} if and only if it does not admit any rational first integrals, that is, if $f\in\mathbb C(P)$ such that $vf = 0$ for all $v\in \mathcal F$ then $f \in \mathbb C$. From Theorem 1.4 in \cite{bonnet} we have that $\mathcal F$ is irreducible if and only if it has a dense leaf. Let $\mathcal L$ be leaf of $\mathcal F$ and $\overline{\mathcal L}$ its Zariski closure. Then $\overline{\mathcal L}$ is an irreducible $\mathcal F$-invariant variety and the restricted foliation $\mathcal F|_{\overline{\mathcal L}}$ is irreducible.

If $\nabla$ is a rational connection on a principal bundle $\pi: P \to Y$ then the space
$$\Gamma^{\rm rat}(\mathcal H_{\nabla}) = \mathbb C(P)\otimes_{\mathbb C(Y)}\nabla(\mathfrak X_Y)$$
of rational $\nabla$-horizontal vector field is a singular foliation on $P$ if and only if $\nabla$ is flat. Moreover, if $\nabla$ is an algebraic flat connection then $\Gamma^{\rm rat}(\mathcal H_{\nabla})$ is a regular foliation. By abuse of notation $\Gamma^{\rm rat}(\mathcal H_{\nabla})$-invariant varieties are called $\nabla$-invariant varieties and leaves of $\Gamma^{\rm rat}(\mathcal H_{\nabla})$ are called
$\nabla$-horizontal leaves. If $\nabla$ is algebraic and $\mathcal L$ is a $\nabla$-horizontal leaf then the projection $\mathcal L \to Y$ is a topological cover in the usual topology. 

Note that the foliation $\Gamma^{\rm rat}(\mathcal H_{\nabla})$ has non horizontal leaves (called \emph{vertical leaves}) included in the fibers of $P$ at non regular points of $\nabla$. 

\subsection{The Galois group}

Let us consider $\nabla$ an algebraic connection. The intersection of $\nabla$-invariant varieties is $\nabla$-invariant and therefore for each point $p\in P$ there is a minimal $\nabla$-invariant variety $Z$ such that $p\in Z$. 

Let $\mathcal L_p$ be the $\nabla$-horizontal leaf through $p \in P$.
It is clear that $\pi|_{\mathcal L_p}\colon \mathcal L_p\to Y$ is surjective. 
Moreover if $Z\subset P$ is a $\nabla$-invariant variety and $p\in Z$ then the Zariski closure of the leaf, $\overline{\mathcal L_p}\subset Z$. Note that, by the $G$-invariance of the distribution $\mathcal H_\nabla$, we have $\mathcal L_p \cdot g = \mathcal L_{p\cdot g}$.

\begin{lem}
The following are equivalent:
\begin{itemize}
    \item[(a)] $Z$ is a minimal $\nabla$-invariant variety.
     \item[(b)] For any $p\in Z$, $Z$ is the Zariski closure of $\mathcal L_p$.
    \item[(c)] $Z$ is the Zariski closure of a $\nabla$-horizontal leaf.
\end{itemize}
\end{lem}

\begin{proof}
$(a)\Rightarrow(b)$ Assume that $Z$ is minimal, and let $p\in Z$. As $Z$ is $\nabla$-invariant it implies $\mathcal L_p \subset Z$ and therefore $\overline{\mathcal L_p}\subset Z$. We have that $\overline{\mathcal L_p}$ is a $\nabla$-invariant variety, and then by minimality of $Z$ we have $Z = \overline{\mathcal L_p}$.

$(b)\Rightarrow(c)$ Trivial. 

$(c)\Rightarrow(a)$ Let us consider $p\in Z$ such that $Z = \overline{\mathcal L_p}$. Note that $Z$ is irreducible.
Let us see that $Z$ is minimal. Let us consider $W\subset Z$ a $\nabla$-invariant subvariety and $q\in W$. We have,
$$\overline{\mathcal L_q} \subseteq W \subseteq \overline{\mathcal L_p} = Z.$$
Let us consider $p'\in \mathcal L_p$ in the same fiber than $q$. Note that there is $g\in G$ such that $q\cdot g = p'$. Then we have $\overline{\mathcal L_q}\cdot g = \overline{\mathcal L_{p'}} =\overline{\mathcal{L}_{p}} = Z$. If follows that $Z$ and $W$ have the same dimension. They are irreducible and therefore they are equal.
\end{proof}

\begin{lem}
Let $Z$ be a minimal $\nabla$-invariant variety for an algebraic connection. Then,
$${\rm Gal}(Z) = \{g\in G\colon Z\cdot g = Z\}$$
is an algebraic subgroup of $G$, and $\pi|_{Z}\colon Z\to Y$ is a ${\rm Gal}(Z)$-principal bundle. 
\end{lem}

\begin{proof}
We just need to note that the isomorphism,
$$P\times G \xrightarrow{\sim} P\underset{Y}{\times} P, \quad (p,g)\mapsto (p, p\cdot g)$$
maps $Z\times {\rm Gal}(Z)$ onto $Z\underset{Y}{\times}Z$.
\end{proof}

Note that if $Z$ is a minimal $\nabla$-invariant subvariety  of $P$ then any other is of the form $Z\cdot g$ for some $g\in G$ and 
${\rm Gal}(Z\cdot g) = g^{-1}{\rm Gal}(Z) g$. It follows that $P$ is the disjoint union of minimal $\nabla$-invariant varieties, each one of them a principal bundle, all of them modeled over conjugated subgroups of $G$.

\begin{defn}
The {\em Galois group}  ${\rm Gal}(\nabla)$ of $\nabla$ is the algebraic group ${\rm Gal}(Z)$ for any minimal $\nabla$-invariant variety $Z$. It is a well defined abstract algebraic group, but its embedding as a subgroup of $G$ depends on the choice of $Z$.
\end{defn}

\begin{exam}
Let us consider example \ref{ex:linear}. Then $\mathcal L_p$ is the analytic subvariety obtained by analytic continuation of the germ of a solution of the linear equation with initial condition $p \in Y \times G$ along any path in $Y$ starting at $\pi(p)$.
The differential field $(\mathbb C(\overline{\mathcal L_p}),\nabla_{\frac{d}{dy}})$ 
is a Picard-Vessiot extension of $(\mathbb C(Y), \frac{d}{dy})$. We usually choose a point of the form $p= (y_0, {\rm id}) \in Y \times G$, in which case ${\rm Gal}(\nabla)$ is also called the Picard-Vessiot group at $y_0$.
\end{exam}

\begin{rem}
    The choice of a leaf $\mathcal L$ and a point $y_0 \in Y$ give rise to a representation of the fundamental group $\pi_1(Y,y_0) \to G$ by analytic continuation. Its image, the monodromy group $\Gamma$, in included in ${\rm Gal}(\overline{\mathcal L})$. To understand how big $\Gamma$ is inside $G$, we need to extend rationally $\nabla$ on $\overline{Y}$ a compactification of $Y$ with $\overline{Y}\setminus Y$ a normal crossing divisor. 

    When $\dim_{\mathbb C}(Y) =1$ and $\nabla$ as only regular singularities, Schlesinger's theorem (See  \cite[vol. 2.1 page 101]{Schlesinger} or \cite[\S \, 5.1.2]{van der Put-Singer} in the case $\overline{Y} = \mathbb{CP}_1$) insures that $\Gamma$ is Zariski dense in $G$.

    When $\dim_{\mathbb C}(Y) >1$ and $\nabla$ as only regular singularities, we have no elementary reference for an analogue of Schlesinger's theorem but the Riemann-Hilbert correspondence of Deligne \cite[Th\'eor\`eme 5.9]{deligne} together with his Tannakian description of the differential Galois group \cite[\S 9]{Tannaka} gives the density of the monodromy group in the differential Galois group.

     Without the regular-singular hypothesis, the monodromy is no longer Zariski dense in the Galois group (the exponential function whose differential equation has an irregular singularity is the simplest such example). The situation is much more delicate both from the geometrical context providing such connections (see \cite{Sabbah}) or from the differential Galois theory perspectives (see Ramis's density theorem \cite[Theorem 8.10]{van der Put-Singer}).
    
\end{rem}

\subsection{Connection form}\label{CartanForm} An alternative way to encode a principal connection $\nabla$ is through its connection form $\Omega$. First, let us define the structure form $\omega$, which is canonically attached to the principal bundle. We differentiate the action $R$ of $G$ on $P$, we obtain a trivialization
$$d_2R\colon P \times \mathfrak g \xrightarrow{\sim} T(P/Y).$$
Such trivialization defines the structure form $\omega = {\rm pr}_2 \circ d_2R^{-1}$ of the bundle,
$$\omega \colon T(P/Y) \to \mathfrak g.$$
Note that if $g\in G$ and $F$ is a gauge transformation then $dR_g$ and $dF$ map the vertical bundle $T(P/Y)$ onto itself. Therefore $R_g^*(\omega)$ and $F^*(\omega)$ are well-defined $\mathfrak g$-valued one-forms on $T(P/Y)$. The structure form $\omega$ has the following properties:
\begin{enumerate}
    \item Right $G$-covariance: for $g \in G$, $R_g^\ast \omega = {\rm Adj}_{g^{-1}}\circ \omega$;
    \item Left gauge-invariance: for any gauge transformation $F^*(\omega) = \omega$.
    \item For each fiber $P_y$ the form $\omega_y = \omega|_{P_y}$ satisfy the Maurer-Cartan structure equation:
    $$d\omega_y = -\frac{1}{2}[\omega_y,\omega_y].$$
\end{enumerate}

Given a principal connection $\nabla$ there is a unique way to extend the structure form $\omega$ to a $\mathfrak g$-valued $1$-form $\Omega$ on $P$ that vanish along the horizontal distribution, the so-called Cartan connection form: 
$$\Omega \colon TP \to \mathfrak g, \quad v_p \mapsto \omega(v_p - \nabla_{d\pi(v_p),p}).$$ 
It is clear that $\Omega$ and $\nabla$ determine each other as $\mathcal H_{\nabla} = \ker(\Omega)$. The gauge-invariance property of $\omega$ extends partially to $\Omega$.

\begin{defn}
We say that a gauge transformation $F\colon P\to P$ is a {\em gauge symmetry} of $\nabla$ if for any $p\in P$ and $v\in T_{\pi(p)}Y$, we have that $dF(\nabla_{v,p}) = \nabla_{v,F(p)}$.
\end{defn}
Summarizing, the connection form $\Omega$, attached to a flat principal connection has the following properties:
\begin{enumerate}
    \item $\Omega|_{T(P/Y)} = \omega$.
    \item $\Omega = 0$ on the horizontal distribution
    \item $\Omega$ is $G$-covariant ;
    For $g \in G$, $g^\ast \Omega = {\rm Adj}_{g^{-1}}\circ \Omega$ (right covariant);
    \item $\Omega$ is gauge-invariant: for any gauge symmetry $F$ of $\nabla$, $F^*\Omega = \Omega$ (left invariant)
    \item $\Omega$ satisfies the Maurer-Cartan structure equation $$d\Omega = -\frac{1}{2}[\Omega, \Omega].$$
\end{enumerate}

\begin{exam}\label{linearForm}
Going on with example \ref{ex:linear}, the structure form is,
$$\omega =(U^{-1}dU)|_{T(P/Y)}$$
and the connection form is:
$$\Omega = U^{-1}dU - U^{-1}AUdt.$$ 
\end{exam}

\begin{prop}\label{prp:connection}
Let $\mathcal L$ be a $\nabla$-horizontal leaf in $P$. The restriction of $\Omega$ to the Zariski closure $\overline{\mathcal L}$ of this leaf takes values in ${\rm Lie}({\rm Gal}(\overline{\mathcal L}))$.
\end{prop}

\begin{proof}
Note that $\overline{\mathcal L}\to Y$ is a principal ${\rm Gal}(\overline{\mathcal L})$ bundle. Therefore,
$\omega|_{T(\overline{\mathcal L}/Y)}$ takes values in ${\rm Lie}({\rm Gal}(\overline{\mathcal L}))$. If follows that $\Omega|_{\overline{\mathcal L}}$ takes values in ${\rm Lie}({\rm Gal}(\overline{\mathcal L}))$.
\end{proof}

\section{$\nabla$-Special subvarieties and Ax-Schanuel}

\begin{defn}
Let $\nabla$ be a flat $G$-principal connection over $Y$ with Galois group $G$.
A subvariety $X \subset Y$ is {\em $\nabla$-non-generic} if for each irreducible component $X_i$ with smooth locus $X_i^\ast$ the group ${\rm Gal}(\nabla|_{X_i^{\ast}})$ is a strict subgroup of $G$. Maximal irreducible $\nabla$-non-generic subvarieties are called {\em $\nabla$-special}.

Let $H\subset G$ be an algebraic subgroup. A $\nabla$-non-generic subvariety $X \subset Y$ is {\em $H$-non-generic} if ${\rm Gal}(\nabla|_{X^\ast}) \subset H$. Maximal irreducible $H$-non-generic subvarieties are called {\em $H$-special}.
\end{defn}

\begin{rem}
If $X=\{x\}$ is a point then a horizontal leaf of $\nabla|_X$ is a point in the fiber $P_x$ and the Galois group is just the identity. Points are $\nabla$-non-generic as soon as $Y$ is not a point itself. 
\end{rem}

 \begin{exam}
The connection on the trivial $(\mathbb C^\ast)^2$- bundle over $\mathbb C^2 = Y$ given by 
$$
dU - \begin{bmatrix}  d(y_2y_1) & 0 \\ 0 & d(y_1) \end{bmatrix} U =0
$$
with solution $U =   \begin{bmatrix}  e^{y_2y_1} & 0 \\ 0 & e^{y_1} \end{bmatrix}$ has special subvarieties. Its Galois group is $\mathbb C^\ast \times \mathbb C^\ast$ but its restriction to lines $y_2 = q \in \mathbb Q$ or to lines $y_1 = c \in \mathbb C$ have Galois groups $\mathbb C^\ast$. {The $\nabla$-special} subvarieties are described by Ax's Theorem (cf. \cite{Ax71}) : they are given by equations $\alpha_1 y_1y_2 + \alpha_2 y_1 = \beta$ with $\alpha_1,\alpha_1 \in \mathbb Q$ and $\beta \in \mathbb C$.
 \end{exam}

The relation between the $\nabla$-special subvarieties and special Shimura subvarieties (Hodge type locus) is not clear from the definitions. 

Throughout, unless otherwise stated, we use {\em special} subvarieties as short for $\nabla$-special subvarieties. This conflicts with the usual terminology for variations of mixed Hodge structures: in this case our $\nabla$-special subvarieties coincide with the monodromy special, or weakly special, subvarieties, among which the special subvarieties are the ones with smaller generic Mumford-Tate group. We refer to K. Chiu's paper \cite{Chiu0} for more details.\\

Lie algebras of algebraic subgroups of $G$ are termed \emph{algebraic} Lie subalgebras of ${\rm Lie}(G)$. Given any Lie subalgebra $\mathfrak h\subset {\rm Lie}(G)$ there exists a smallest algebraic Lie subalgebra containing $\mathfrak h$, its algebraic envelope, denoted by $\bar{\mathfrak h}$. 

\begin{defn}\label{def:NZDS}
A complex algebraic group $G$ is said to be {\it sparse} if for any proper Lie subalgebra  $\mathfrak h \subset {\rm Lie}(G)$, the algebraic envelope $ \bar{\mathfrak h}$ is a proper algebraic Lie subalgebra of ${\rm Lie}(G)$.
\end{defn}

The Zariski closure of a proper connected analytic subgroup $H$ of a sparse group $G$ is always a proper algebraic subgroup $\bar{H}\subset G$.

\begin{exam}
In general, any Lie subalgebra $\mathfrak h\subseteq {\rm Lie(G)}$ is an ideal of its algebraic envelope $\bar{\mathfrak h}$. It follows that semi-simple groups are sparse.
\end{exam}

The first version of our Ax-Schanuel theorem is the following.

\begin{thm}[\bf Theorem \ref{thmA}]\label{th:AxSh1}
Let $\nabla$ be a $G$-principal connection on $P \to Y$ with sparse Galois group $G$.
 Let $V$ be an algebraic subvariety of $P$, $\mathcal L$ an horizontal leaf and $\mathcal V \subset V\cap \mathcal L$ an irreducible analytic subvariety of $P$. If $\dim V < \dim (\mathcal V) + \dim G$
 then the projection of $\mathcal V$ in $Y$ is contained in a $\nabla$-special subvariety.
\end{thm}

\begin{rem}
For a $G$-principal connection on $P \to Y$ with $\dim Y =n$, an horizontal leaf $\mathcal L$ is an injectively immersed analytic submanifold, and locally $(\mathcal L, p)$ is biholomorphic to $(\mathbb C^n,0)$. For an algebraic subvariety $V \subset P$ and a point $p \in \mathcal L \cap V$, we have that $\mathcal L \cap V$ in $(\mathcal L, p)$ is a germ of analytic subvariety. Then in Theorem \ref{th:AxSh1} one can choose $\mathcal V$ to be an irreducible component of $\mathcal L \cap V$ seen as a germ of analytic subvariety of $(\mathcal L,p)$.

Given a formally parameterized space $\hat{\mathcal V} : {\rm Spf}\:\mathbb C[[s_1, \ldots, s_m ]] \to \mathcal L \subset P$, we let $\mathcal V$ be the smallest germ of analytic subvariety containing it and $V$ the Zariski closure of $\mathcal V$. In particular one gets the inclusions:
$\hat{\mathcal V} \subset \mathcal V \subset \mathcal L \cap V$. 
\end{rem}

From the remark above one gets another formulation of Ax-Schanuel type theorem involving formal power series.

\begin{cor}\label{FirstAx}
Let $\hat{\mathcal V} : {\rm Spf}\:\mathbb C[[s_1, \ldots, s_n ]] \to \mathcal L \subset P$ be a non constant formally parameterized space in a horizontal leaf $\mathcal L$ of $\nabla$ and $V$ its Zariski closure.
 If $\dim V < {\rm rk}(\hat{\mathcal V}) + \dim G$ then the projection of $V$ in $Y$ is contained in a special subvariety.
 \end{cor}

\subsection{Proof of Theorem \ref{th:AxSh1}}

If $\mathcal V$ is a point then its projection is a point thus is special. Let us assume $\dim \mathcal V>0$.
Then the proof of the theorem follows from the next lemmas about the connection form $\Omega$. We can assume without loss of generality that $V$ is the Zariski closure of the component $\mathcal V$ of $V \cap \mathcal L$, in particular $V$ is irreducible.

\begin{lem}\label{rank}
The restriction of $\Omega$ to $V$ has rank strictly smaller than $\dim G$ at the generic point of $V$.
\end{lem}
\begin{proof}

 As the function on $V$ defined by $p \mapsto \dim \ker \Omega|_{T_p V}$ is upper semi-continuous, the set of points $p \in V$ such that the dimension of the kernel of $\Omega|_V$ at $p$ is greater than or equal to $\dim \mathcal V$ is a Zariski closed subset. This set contains the points of $\mathcal V$ for which the tangent vectors are in the kernels. This analytic space is Zariski dense in $V$. It hence follows that the rank of $\Omega|_V$ is smaller than $\dim V - \dim \mathcal V < \dim G$. As $\dim \mathcal V >0$, the lemma is proved.
\end{proof}

\begin{lem}
There is a proper Lie subalgebra $\mathfrak{h}\subset\mathfrak{g}$ such that for all $p\in V$, $\Omega_p(T_pV) \subset \mathfrak h$ with equality for generic $p$.
\end{lem}
\begin{proof}
We show that $\Omega_p(T_p V)$ does not depend on $p$ in a Zariski open subset of $V$. Let $e_1,\ldots,e_q$ be a basis of $\mathfrak{g}$ and decompose
$$ \Omega|_V = \sum_{i=1}^q \Omega_i e_i.$$

 We may assume that the first $\Omega_1,\ldots,\Omega_k$ form a maximal set of linearly independent 1-forms over $\mathbb C(V)$ among the $\Omega_i$'s. We have then:
 
 $$\Omega_{k+i} = \sum_{j=1}^k b_{ij}\Omega_j$$ 
with $b_{ij}$ rational functions on $V$. We consider a vector field $D$ on $V$ in the kernel of $\Omega$. By taking Lie derivatives we obtain:
$${\rm Lie}_D (\Omega_{k+i}) = \sum_{j=1}^k (D\cdot b_{ij})\Omega_j + \sum_{j=1}^k b_{ij}{\rm Lie}_D\Omega_j.$$

 From the Cartan formula (${\rm Lie}_D = i_D\circ d + d\circ i_D$) and Cartan structural equation ($d\Omega_j$ is a combination of the $2$-forms $\Omega_k\wedge\Omega_\ell$ with constant coefficients) we have that ${\rm Lie}_D\Omega_j = 0$ for any $j=1,\ldots,q$, and therefore for $i=1,\ldots, q-k$ we have 
 $$\sum_{j=1}^k (D\cdot b_{ij})\Omega_j = 0.$$
 As the $\Omega_i$'s, $1\leq i\leq k$, are linearly independent, we obtain that $D\cdot b_{ij}=0$ for all $1\leq i\leq q-k$ and $1\leq j\leq k$. It hence follows that the $b_{ij}$'s are first integrals of the foliation of $V$ defined by $\ker(\Omega|_V)$ and thus are constant on $\mathcal V$. Since $V$ is the Zariski closure of $\mathcal V$, the functions $b_{ij}$ are constant. We have thus proved that for $p$ a point of maximal rank of $\Omega|_{T_pV}$, the image of $T_p(V)$ by $\Omega_p$ is a fixed linear subspace $\mathfrak h \subset \mathfrak g$ independent of $p$.
 
 We claim that $\mathfrak h$ is a Lie subalgebra. Indeed, let $e_1$, $e_2$ be two elements of $\mathfrak h$ and let $v_1$ and $v_2$ be two vector fields on $V$ such that $\Omega|_V(v_i)= e_i$ for $i=1,2$. Then by Cartan's structural equation, 
 $$[e_1,e_2] = d\Omega(v_1,v_2) $$
 
 On the other hand, as $\Omega$ is a $1$-form with values in $\mathfrak h$, we have that $d \Omega$ is a $2$-form with values in $\mathfrak h$.
\end{proof}

Let us consider $\bar H$, the connected algebraic subgroup of $G$ whose Lie algebra is $\bar{\mathfrak h}$. Note that since $\mathfrak h$ is a proper Lie subalgebra of ${\rm Lie}(G)$ and $G$ is sparse, it follows that $\bar H$ is a proper algebraic subgroup of $G$. Let us also consider the quotient of $P$ by the action of $\bar H$, namely $\rho : P \to P/\bar H$ and $\ell : \mathfrak g \to \mathfrak g/\bar{\mathfrak h}$. The composition $\ell \circ \Omega$ is a ${\rm Lie}(G)/\bar{\mathfrak h}$ valued $1$-form on $P$. From the Cartan structure equation, we have that $\rm ker(\ell\circ\Omega)$ is a foliation on $P$. Let us recall that a foliation $\mathcal G$ on $P$ is said to be $\rho$-projectable if there is a foliation $\mathcal F$ on $P/\bar H$ whose leaves are the projections by $\rho$ of leaves of $\mathcal G$.

\begin{lem}
The foliation $\mathcal G$ on $P$ defined by $\ker (\ell \circ \Omega)$ is $\rho$-projectable on a foliation $\mathcal F$ on $P/\bar H$ of rank equal to the dimension of $Y$.
\end{lem}

\begin{proof}
If $v$ is an infinitesimal generator of the action of $\bar H$ then  $\Omega(v) \in \bar {\mathfrak h}$ and $\ell \circ \Omega(v)=0$. Thus leaves of $\mathcal G$ are saturated by orbits of $\bar H$. This implies that $\mathcal G$ is $\rho$-projectable. Let $\mathcal F$ be the projection of $\mathcal G$ on $P/\bar H$.

As $\Omega$ is onto and $\dim \ker \ell = \dim \bar H$, the vector fields in $\ker \Omega$ and the infinitesimal generators of $\bar H$ are local generators of $\mathcal G$. For such a infinitesimal generator $v$ we have that $d\rho (v)=0$ and hence $ {\rm rank}\,  \mathcal F \leq \dim Y$. The projection on $Y$ of vectors in $\ker(\Omega)$ is injective and commutes with $d\rho$. Hence $ {\rm rank}\,  \mathcal F \geq \dim Y$.
\end{proof}

We can now conclude the proof of Theorem \ref{th:AxSh1}.
\begin{proof}
The projection of $\rho(V)$ in $Y$ is a constructible set. Let $X$ be its Zariski closure. To prove that $X$ is $\nabla$-non-generic, it suffices to build a strict $\nabla|_X$-invariant subvariety of $P|_X$.

As $\ell \circ \Omega|_V = 0$, we have that $V$ is included in a leaf of $\mathcal G$. Thus $\rho(V)$ is included in a leaf of $\mathcal F|_X$ and we have that $\dim \rho(V) \leq {\rm rank}\, \mathcal F|_X$. By construction $\dim \rho(V) \geq \dim X$ and $ {\rm rank}\,  \mathcal F|_X = \dim X$. So it follows that $\dim \rho(V) = {\rm rank}\, \mathcal F|_X$. From this we get that $\rho(V)$ is an algebraic leaf of $\mathcal F |_X$ and this implies that $\rho^{-1}(\rho(V))$ is a  $\nabla|_X$-invariant subvariety whose fibers are finite union of $\bar H$ orbits. So the Lie algebra of ${\rm Gal}(\nabla|_X)$ is included in $\bar{\mathfrak h}$ and thus $X$ is $\nabla$-non-generic.
 \end{proof}

\subsection{Uniformizing equation and Ax-Schanuel Theorems}\label{sectionUniformization}
  
Let $G$ be an algebraic group and $B$ an algebraic subgroup. A $(G, G/B)$-structure on an algebraic variety $Y$ is usually defined using charts on $Y$ with values in $G/B$ and change of charts in $G$. Here is an algebraic version of this notion.

The jet space $J^\ast(Y, G/B)$ of invertible jets of maps from $Y$ to $G/B$ is endowed with an action of $G$ by postcomposition. As it is a jet space, its ring has a $\mathcal D_Y$-differential structure, see \cite[\S 2.3.2]{Beilison-Drinfeld}, where $\mathcal D_Y$ is the non-commutative algebra of linear differential operators on $\mathcal O_Y$.  

\begin{defn}
A {\em rational $(G, G/B)$-structure} on $Y$ is a $\mathcal D_Y$-subvariety $\mathscr{C}$ of $J^\ast(Y, G/B)$ such that $\mathscr C|_{Y^o}$ is a $G$-principal sub-bundle of $J^\ast(Y^o, G/B)$, for some dense Zariski-open subset $Y^o \subset Y$.
\end{defn}

The set of solutions of $\mathscr{C}$ denoted also by $\mathscr C$ is the set of holomorphic maps $c : U \to G/B$ defined on an Euclidean open subset $U \subset Y$ whose Taylor expansions $j^\infty(c) : U \to J^\ast(Y, G/B)$ take values in $\mathscr C$. This is the set of charts of the $(G,G/B)$-structure.

As $\mathcal D_Y$-varieties have no $ \mathcal O_Y$-torsion, $\mathscr C$ is well-defined if we know $\mathscr C|_{Y^o}$. The following lemma explains that the Theorem \ref{thmA} can be applied to study some properties of $(G, G/B)$-structures.

\begin{lem}\label{GGB-connection}
A rational $(G, G/B)$-structure $\mathscr C$ on an algebraic variety $Y$ defines a $G$-principal connection on some Zariski open subset $Y^o$.
  \end{lem}

\begin{proof}
By definition we have that $\mathscr{C}$ is a $G$-principal bundle over $Y^o$. The $\mathcal D_Y$-structure of $\mathscr{C}$ gives a lift of vector fields on $Y^\circ$ to $\mathscr{C}$, it is a connection on $\mathscr{C}|_{Y^\circ}$. 
The group $G$ acts on $\mathscr{C}$ by post-composition. The action of $\mathcal D_Y$ is the infinitesimal part of the action by precomposition. These two commute and hence the connection is $G$-invariant.   
\end{proof}

\begin{rem}
The natural action of $G$ on $J(Y,G/B)$ is a left action but the tradition in differential geometry is to endow principal bundles with a right action. To be exactly in the framework of Section \ref{cartan}, we should change the action into a right action. We will not mention this in the sequel.

If $(G, G/B)$-structure are a particular case of principal bundle, not all principal bundle can be presented in this way. For instance if there exists no $B \subset G$ such that $\dim_{\mathbb C} G/B = \dim_{\mathbb C} Y$ the trivial principal bundle $Y \times G$ can not be the principal bundle of a $(G, G/B)$-structure.
\end{rem}

\begin{rem}
    For any chart $c \in \mathscr C$ defined on $U$, the graph of its Taylor expansion $j^{\infty}(c) :U \to J(Y,G/B)$ is an horizontal leaf of the principal connection given by Lemma \ref{GGB-connection}. Any two charts defined on the same open subset are related by post composition with an element of $G$. Hence a rational $(G, G/B)$-structure on $Y$ defines a usual $(G, G/B)$-structure on a dense Zariski open subset $Y^\circ \subset Y$. The charts do not extend through $Y-Y^\circ$ but the differential equation satisfied by the charts can be extended rationally on $Y$. For this reason rational  $(G, G/B)$-structures can be seen as singular $(G, G/B)$-structures with tame singularities.
    
  The converse is false. Consider the punctured affine complex line $\mathbb C \setminus \{0\}$. Solutions of the differential equation $c'(x) = \exp(1/x) $ are charts of a $(\mathbb G_a, \mathbb A^1)$ structure (i.e. a translation structure) on $\mathbb C \setminus \{0\}$. As the set of charts is not the solution set of a rational differential equation, the structure is not a rational one.
    
\end{rem}

Since $\dim Y = \dim G/B$, one has an isomorphism from $J^\ast(Y, G/B)$ to  $J^\ast(G/B, Y)$. If $\mathscr C$ is a $(G, G/B)$-structure, then its image under this isomorphism is denoted by $\mathscr{Y}$ and is a finite dimensional (over $\mathbb C$) $\mathcal D_{G/B}$-subvariety of $J^\ast(G/B, Y)$. 
A local analytic solution $\upsilon$ of $\mathscr Y$ is a holomorphic map defined on a Euclidean open subset $U \subset G/B$, i.e., $\upsilon : U \to Y$, whose Taylor expansion $j^\infty(\upsilon): U \to J^\ast(G/B, Y)$ takes values in $\mathscr Y$. These solutions are called uniformizations of the  $(G, G/B)$-structure.

We call $\mathscr Y$ the space of uniformizations of the  $(G, G/B)$-structure $\mathscr C$. As the data $\mathscr C$ and $\mathscr Y$ are equivalent, we will use any of them to denote the $(G,G/B)$-structure.

The following statement is another (more usual) version of the Ax-Schanuel theorem. It follows by applying Corollary \ref{FirstAx} to $\mathscr{C}$.

\begin{cor}\label{SecondAx}
Let $\upsilon $ be a uniformization of a $(G, G/B)$-structure on a $m$-dimensional algebraic variety $Y$ with $G$ sparse. Assume that $\hat \gamma$ is a non constant formal curve on $G/B$ such that 
$$
{\rm tr.deg.}_\mathbb C \mathbb C \left(\hat \gamma, (\partial^\alpha \upsilon)(\hat \gamma) : \alpha \in \mathbb N^m \right) < 1 + \dim G
$$
then the Zariski closure  $\overline{\upsilon(\hat \gamma)}$ is 
contained in a special subvariety of $Y$.
  \end{cor}
  
From Corollary \ref{SecondAx} we get the Ax-Schanuel theorem without derivatives as stated for instance in \cite{ASmpt}.

\begin{cor}[\bf Corollary B]\label{thirdAx}
Let $\upsilon$ be a uniformization of an irreducible $(G, G/B)$-structure on an algebraic variety $Y$ with $G$ sparse.
Assume $W \subset G/B \times Y$ is an irreducible algebraic subvariety intersecting the graph of $\upsilon$. Let $U$ be an irreducible component of this intersection such that
$$
\dim W < \dim U + \dim Y.
$$
Then the projection of $U$ to $Y$ is contained in a special subvariety of $Y$.
\end{cor}
\begin{proof}
Let $j^0 : \mathscr{C} \to G/B \times Y $ be the $0$-jet projection. Then $(j^0)^{-1}(W)$ intersects the graph of the jet of $\upsilon$ on $\tilde U$ such that $j^0(\tilde U) = U$.

Now $\dim ({(j^0)}^{-1}(W)) = \dim W + \dim B < \dim U + \dim X + \dim B = \dim \tilde U + \dim G$. Using Corollary \ref{SecondAx} the result follows.
\end{proof}

We also obtain the following.
\begin{cor}
Let $\upsilon$ be a uniformization of a $(G, G/B)$-structure on an algebraic variety $Y$ with $G$ sparse.
Let $A \subset G/B$ be an irreducible algebraic subvariety. If $\upsilon(A)$ has a proper Zariski closure in $Y$, then  $\upsilon(A)$ is contained in a special subvariety of $Y$.
\end{cor}

\begin{proof}
Let $U \subset G/B \times Y$ be the graph of $\upsilon|_A$ and $V$ be its Zariski closure. Then
$$
\dim V \leq \dim A + \dim \overline{\upsilon(A)} < \dim U + \dim X.
$$\end{proof}

\section{Differential equations and covering maps}\label{coveringMap}

Our main results of the previous section have applications to algebraicity problems for certain analytic functions which are covering maps of quotients of complex analytic spaces by discrete groups; including, for instance, Shimura varieties. There is a subtle difference in the setup to these results from our main theorems of the previous section. In both the Shimura-type settings as well as Scanlon's generalization \cite{Scanlon}, one starts with additional data compared to our setup from the previous section. In all of the covering map situations, one has an open subset $U$ of a complex algebraic variety admitting an action of an algebraic group $G$ such that $G(\m R)$ acts by isometries on $U$. The covering maps considered realize the quotient of $U$ by a discrete subgroup $\Gamma \leq G(\m R)$ as an algebraic variety. Many of these results fit into our framework from the previous section - that is, $\Gamma$ is a discrete Zariski dense subgroup of a sparse algebraic group. 

But as we have noted in the introduction, rather few sparse (or even simple) algebraic groups $G$ arise in this way for Shimura varieties or, as far as we know, in Scanlon's setup. We next give a concrete example to which our results from the previous section apply, but which cannot be put in the setting of covering maps. 

\begin{exam}\label{Appell} Consider Appell's bivariate hypergeometric system $F_1$. It is a rank 3 linear connection on a vector bundle on $Y = \mathbb{CP}_1 \times \mathbb{CP}_1$ with singularities along 7 lines: $\{0,1,\infty\} \times \mathbb{CP}_1$, $\mathbb{CP}_1 \times \{0,1,\infty\}$ and the diagonal. This connection gives a rational $({\rm PSL}_3(\mathbb C), \mathbb{CP}_2)$-structure on $Y$ called $\mathscr C_{hyp}$ or $\mathscr Y_{hyp}$ (see \cite[Chapter 6]{Yoshida}).
 
 Only for very particular and clever choices of the exponents in the $F_1$ system (cf. \cite[Theorem 11.4]{delignemostow} are solutions of $\mathscr Y_{hyp}$ built from the quotient of the ball $\mathbb B \subset \mathbb{CP}_2$ by a lattice $\Gamma \subset {\rm PSU}(2,1) \subset {\rm PSL_3(\mathbb C)}$. That is, for most choices of the system $F_1$, the example does not fit into the context of the covering maps but our results still applies whenever ${\rm Gal}(\mathscr C_{hyp})={\rm PSL}_3(\mathbb C)$.
\end{exam}

In subsection \ref{scanloncov} we explain how our results apply to the setting of Scanlon \cite{Scanlon}. In particular, our results give a solution of Problem 6.2 of \cite{Scanlon} in the case of sparse groups $G$. Following this, we give short expositions of how our results apply to the settings of modular curves \ref{covcurves}, Shimura varieties \ref{covshim}, and ball quotients \ref{covball}. These sections are somewhat expository in the sense that they are special cases of subsection \ref{scanloncov}. 

\subsection{From covering maps to $(G,G/B)$-structures} \label{scanloncov}

In this section, we follow much of the notation and assume the background of \cite{Scanlon}, where much of the work is stated in terms of classical differential algebra. We remind the reader of some of these terms next. Scanlon's article contains numerous examples which were classically known, and we refer the reader to the discussion of the history of such results in \cite{Scanlon}. 

The Kolchin topology over a differential field $F$ is the differential version of the Zariski topology with Kolchin closed sets of affine $n$-space corresponding bijectively to differential radical ideals of the ring $F \{x_1, \ldots , x_n \}$ of differential polynomials over $F$. A function is differential algebraically constructible if its graph can be expressed as a finite Boolean combination of closed sets in the Kolchin topology.

Let $G$ be a complex algebraic group, $B \subset G$ an algebraic subgroup and $X=G/B$ the associated homogeneous space. Let $U$ be an open complex submanifold of $X(\m C)$ and $ \Gamma $ be a Zariski dense subgroup of $G (\m C)$ with the property that the induced action of $\Gamma$ on $X$ preserves $U$. Assume that we have a complex analytic map $\upsilon :U \to Y$ which is a covering map of the complex algebraic variety $Y$ expressing $Y(\m C)$ as $\Gamma \setminus U$. In this section, we show that the determinations of $\upsilon^{-1}$ are charts of an algebraic $(G,G/B)$-structures. The monodromy of this locally defined function is the group $\Gamma$.  
Under those assumptions Scanlon shows that there is a differential algebraically constructible function $\tilde\chi: X \to Z$, 
for some algebraic variety $Z$, called the {\em generalized Schwarzian derivative} associated to $\upsilon$ such that for any differential field $F$ having field of constants $\m C$ and points $a,b\in X(F)$ one has that $\tilde\chi(a)=\tilde\chi(b)$ if and only if $a=gb$ for some $g\in G(\m C)$. For the situation we have in mind, the restriction to differential fields $F$ with $m : = \dim X$ commuting derivatives is enough. One should note that in many instances, the map $\tilde\chi$ was already known in the literature.

From Scanlon's construction, it follows that the generalized Schwarzian derivative is defined on an order $k$ jet space\footnote{For the definition of jet spaces, see section 2.1 of \cite{Scanlon}; roughly speaking, these are affine bundles of a complex manifold which are higher order tangent spaces. Here  $J^\ast_{k,m}(X)$ refers to the space of invertible $k$-jets of holomorphic maps from $(\mathbb C^m,0)$ to $X$.}$\tilde\chi : J^\ast_{k,m}(X) \to Z$. Thus the map $\chi:=\tilde\chi \circ \upsilon^{-1} : J^\ast_{k,m}(Y) \to Z$ is a well-defined analytic map and induces a differential analytic map $Y(M) \rightarrow Z(M)$ for $M$ any field of meromorphic functions in $m$ variables.

Now, also assume that the restriction of $\upsilon$ to some set containing a fundamental domain is definable in an o-minimal expansion of the reals as an ordered field. One of the main results of \cite[Theorem 3.12]{Scanlon} is that under this assumption the function $\chi:=\tilde\chi \circ \upsilon^{-1}$, for any choice of a branch of $\upsilon^{-1}$, is also differential algebraically constructible. 

The function $\chi$ is called the {\em generalized logarithmic derivative} associated to $\upsilon$.

\begin{exam}
Consider the case when $G={\rm PSL}_2$ and $B$ its subgroup of lower triangular matrices so that $X(\mathbb{C})=\mathbb{CP}_1$. Let $U$ be the complex upper half plane $\mathbb{H}$ seen as an open subset of $\mathbb{CP}_1$ and $\Gamma\subseteq{\rm PSL}_2(\mathbb{R})$ a Fuchsian group of first kind. Then $Y$ is an algebraic curve. In this case, it is classically known that $\tilde\chi : J_{3,1}^\ast(X) \to \mathbb C$ can be taken to be the Schwarzian derivative $S(x)=\left(\frac{x''}{x'}\right)' -\frac{1}{2}\left(\frac{x''}{x '}\right)^2$ and $\chi : J_{3,1}^\ast(Y) \to \mathbb C$ is given by $S(y) + R(y)y'^2$ with $R(y) = S(\upsilon^{-1})$.

The definability property of $\upsilon$ in Scanlon's theorem is analogous to the hypothesis on solutions $\upsilon^{-1}$ of the linear differential equation to have moderate growth at singular points. Both implies that $R$ is a rational function. 
\end{exam}

Using $\chi$ and $\tilde\chi$, we can define the differential equation satisfied by $\upsilon$ and $\upsilon^{-1}$. From the composition of jets one gets two maps, $c_1 :J^\ast_{k,m}(X)  \times J^\ast_{k,m}(Y) \to J^\ast_k (X,Y)$, where $c_1(u,v) = v\circ u^{-1}$ and $c_2 : J^\ast_{k,m}(X)  \times J^\ast_{k,m}(Y) \to J^\ast_k (Y,X)$, where $c_2(u,v) = u \circ v^{-1}$. The constructible algebraic subvariety given by $\chi - \tilde \chi = 0$ projects by $c_1$ on a constructible algebraic subvariety of $J^\ast_k (X,Y)$ and by $c_2$ on a constructible algebraic subvariety  of $J^\ast_k(Y,X)$

Let us assume that $\bar t =(t_1, \ldots, t_m)$ are the coordinates on $U \subset X=G/B$ for some algebraic subgroup $B$ of $G$ and $\bar y=(y_1, \ldots, y_m)$ be coordinates on $Y$. By construction, we have that $\upsilon(\bar t)$ satisfies the constructible algebraic differential equation
$$ \mathscr Y = c_1\left(\{\chi(\bar y)-\tilde\chi(\bar t)= 0 \}\right) \subset J^\ast_k (G/B,Y)$$
and the inverse branches $\upsilon^{-1}(\bar y)$ satisfies
$$\mathscr C = c_2\left( \{\chi(\bar y)-\tilde\chi(\bar t)=0\}\right) \subset J^\ast_k(Y,G/B).$$

\begin{rem}
The $\ast$ appearing in the jet space $J^\ast_k$ means that we are adding to the explicit equations above inequalities ensuring that the rank of the Jacobian matrix of solutions is $m$.
\end{rem}

\begin{prop}
The equations above define a rational $(G, G/B)$-structure $\mathscr C$ on $Y$.
\end{prop}

\begin{proof}

Since $\bar y$ are coordinates on $Y$, using $c_2$ it follows that $\chi(\bar y)$ is a rational function $R$ on $Y$. Let $Y^\circ$ be the domain of $R$ on which our differential equation can be written $\tilde \chi (\bar t) = R(y)$ in $J^\ast_k(Y,G/B)$. By construction, it follows that $\mathscr C$ coincides with the differential subvariety of $J^\ast(Y,G/B)$ defined by this equation on the open subset on which the constructible set is an algebraic set.

By Scanlon's construction, we have a map $\mathscr C \times G \to \mathscr C$ ensuring that $\mathscr C$ is a $G$-principal bundle on $Y^\circ$. As a smooth finite dimensional $\mathcal D_{Y^\circ}$-subspace of $J^\ast(Y^\circ,G/B)$, $\mathscr C$ gives rise to a connection. The properties of $\bar \chi$ ensure that this connection is $G$-invariant. 
\end{proof}

It hence follows that our Ax-Schanuel Theorems (corollaries \ref{SecondAx} and \ref{thirdAx}) hold in the case of covering maps given in the Scanlon theory when $G$ is sparse. In the coming subsections, we will describe the settings in which existing versions of an Ax-Schanuel Theorem for covering maps exists and have some overlap with our results. 

\subsection{Modular curves} \label{covcurves}
Let $j: \m H \rightarrow \m A^1(\m C)$ be the classical modular $j$-function. 
In the notation above we take $G={\rm SL}_2$ and $$B=\left\{\begin{pmatrix}
a&c\\
0&d
\end{pmatrix}:ad=1\right\}$$ the subgroup of upper triangular matrices so that $X={\rm SL}_2(\m C)/B\cong\mathbb{CP}_1$. As well-known, if we take $\Gamma={\rm SL}_2(\m Z)$, then the quotient $Y(1)=\Gamma \backslash  \m H$ can be identified with the affine line $\m A^1(\m C)$. We take $U$ to be the open subset of $\m H $ such that $j:U \rightarrow \m A^1(\m C)\setminus\{0,1728\}$ is a covering map. 

The restriction of $j$ to the domain $F=\{z\in\m H\;:\;|{\rm Re}(z)|\leq\frac{1}{2}\;\text{ and }\;{\rm Im}(z)\geq\frac{\sqrt{3}}{2}\}$, which contain a fundamental domain, is definable in $\m R_{an,exp}$. Hence $j$ is a solution to a $(G,G/B)$-structure as describe above. In this case, this structure can be taken to be the well-known Schwarzian differential equation satisfied by $j$.

In this setting, the Ax-Schanuel Theorem is a result of Pila and Tsimerman \cite[Theorem 1.1]{AXS}

\begin{thm}
Let $V \subset (\mathbb{ C P}_1 \times Y(1))^n$ be an algebraic subvariety, $D$ the Cartesian product of graph of $j$ in $(\mathbb{ CP}_1\times Y(1))^n$  and $U$ be a component of $V \cap D$ . Then $\dim U = \dim V - n $ unless the projection of $U$ to $Y(1)^n$ is contained in a proper weakly special subvariety of $Y(1)^n.$
\end{thm}

Replacing $\Gamma$ in the above results with $\Gamma(N),$ the kernel of the reduction mod $N$ map ${\rm SL}_2 ( \m Z) \rightarrow {\rm SL}_2 (\m Z /N \m Z)$ (also replacing $j$ with suitable $j_N$ and $Y(1)$ with $Y(N)$), one obtains the same result; $j$ and $j_N$ are interalgebraic as functions over $\m C$, i.e., $\mathbb{C}(j)^{\rm alg}=\mathbb{C}(j_N)^{\rm alg}$. Pila and Tsimerman \cite[Theorems 1.2 and 1.3]{AXS} prove the more general version of Ax-Schanuel which includes the first and second derivatives of $j$ (and replaces $n$ with $3n$). Our work gives a uniform proof of this result for $\Gamma \subset {\rm SL}_2 (\m R)$ any Fuchsian group of the first kind.

\subsection{Shimura Varieties} \label{covshim}  We follow closely the exposition given in \cite{BakkerTsimermann}; in particular our definition of a Shimura variety and of the weakly special subvarieties are set up identically to Section 4.2 of \cite{BakkerTsimermann}. Let $G$ be a connected semi-simple algebraic $\mathbb{Q}$-group and $K$ a maximal compact subgroup of $G(\m R)$ chosen so that $\Omega=G(\mathbb{R})/K$ is a bounded symmetric domain. It is known (cf. \cite[Proposition 7.14]{Helgason}) that the compact dual $\check\Omega$ of $\Omega$ is given as the quotient $\check\Omega=G(\mathbb{C})/B$ for a  parabolic subgroup $B$ and is a homogeneous projective variety. One can always assume that $K\subset B$, so that $\Omega$ is a semi-algebraic subset of $\check\Omega$. 

Given an arithmetic lattice $\Gamma\subset G(\m Q)$, the analytic quotient $Y:=\Gamma \backslash \Omega=\Gamma \backslash G(\mathbb{R})/K$ has the structure of an algebraic variety and is called a pure (connected) Shimura variety. The quotient map $q:\Omega\rightarrow Y :=\Gamma \backslash \Omega$ is a covering map\footnote{To obtain a covering map, one might need to restrict $q$ to an open subset $U$ of $\Omega$ avoiding ramification points of the original map.} and the result \cite[Theorem 1.9]{HyperALW} shows that it is definable in $\m R_{an,exp}$ on some fundamental domain. Hence $q$ is a solution to a $(G, G/B)$-structure on $Y$  as defined above.

We fix $Y = \Gamma \backslash\Omega$ a connected pure Shimura variety and $q:\Omega\rightarrow Y$ the quotient map.
\begin{defn} \label{weaklyspec} \cite[4.2.2]{BakkerTsimermann}
A {\em weakly special subvariety} of $Y$ is a Shimura variety $Y'$ given as
$$Y'=\Gamma' \backslash G'(\mathbb{R})/K'$$
where $G'$ is an algebraic $\m Q$-subgroup of $G$, the group $\Gamma'=\Gamma\cap G'(\m Q)$ is an arithmetic lattice, and $K' = K \cap G'(\m R)$.
\end{defn}
From the definition, $Y'$ is analytic subvariety, but a result of Ullmo and Yafaev \cite{UYspecial} shows that such $Y'$ is algebraic and in fact such $Y'$ are exactly the bialgebraic subvarieties for the uniformization map $q.$ That is, there is an algebraic subvariety $V$ of $\check\Omega$ such that $Y'=q(V\cap\Omega)$ (see Section \ref{comparisonsection} below).

\begin{thm} \label{weaklyspeccomp} \cite[Theorem 1.2]{UYspecial} An irreducible subvariety $Z \subset Y$ is weakly special if and only if some (all) analytic irreducible components of $q^{-1} (Z)$ are algebraic. 
\end{thm}

An essential part of Pila's strategy for attacking various diophantine problems associated with the geometry of certain analytic covering maps (see e.g. \cite{Pila1}) is to identify the (weakly) special subvarieties of $Y$.

In this setting, the most general transcendence result to which our applications are related comes from \cite{ASmpt}. For $k \in \mathbb N$, $k\geq 2$, let $J_k (\check \Omega , Y)$ be the variety of jet of order $k$ of maps from $\check \Omega$ to $Y$, $j^k(q) : \Omega \to J_k (\check \Omega , Y)$ be the section given by Taylor expansions of $q$ and $D$ its graph.

\begin{thm}\cite[Theorem 1.3]{ASmpt} Let $W \subset J_k^\ast (\check \Omega , Y)$ be an algebraic subvariety. Let $U$ be a component of $W \cap D$ of positive dimension. Suppose that $$\dim W < \dim D + \dim G.$$ 
Then the projection of $D$ to $Y$ is contained in a proper weakly special subvariety of $Y$. 
\end{thm}
Notice that this statement is Corollary \hyperlink{corB1}{B$^*$} from the Introduction and follows from Corollary \ref{thirdAx} at once after one knows that any $\nabla$-non-generic subvariety is contained in a proper weakly special subvariety, a result of \cite{Chiu} in the setting of Corollary \hyperlink{corB1}{B$^*$}. For additional discussion of this point, see Section \ref{comparisonsection}.

\subsection{Ball quotients} \label{covball}
In this subsection we outline the setting of the recent manuscript \cite{BUball}. Let $G(\mathbb R):= {\rm PU}(n,1)$, the group of holomorphic automorphisms of the unit ball $\m B^n \subset \m C^n.$ Concretely, let ${\rm U}(n,1)$ denote the group of linear transformations of $\m C^{n+1}$ leaving invariant the form: $$z_1 \bar z_1 + z_2 \bar z_2 + \ldots + z_n \bar z_n - z_0 \bar z_0,$$ 
namely, 
$${\rm U}(n,1) = \left \{ g \in {\rm GL}_{n+1} ( \m C ) \, | \, g^T \left( \begin{matrix}
I_n & 0 \\
0 & -1
\end{matrix} \right) \bar g = \left( \begin{matrix}
I_n & 0 \\
0 & -1
\end{matrix} \right) \right\}.$$ 
For $g \in {\rm U}(n,1),$ define the map $\phi_g : \m B^n \rightarrow \m B^n$ as follows: if 
$$g= \left( \begin{matrix}
A & a_1\\
a_2 & a_0 
\end{matrix} 
\right)$$ where $A$ is a $n \times n $ matrix, $a_1$ is a column vector, and $a_2$ is a row vector, then for $z \in \m B^n,$ $$\phi_g (  z) = \frac{A z + a_1 }{a_2 z +a_0 }.$$ 
It is not hard to show that the map $\phi_g$ is the identity if and only if $g \in \{ e^{it} I_{n+1}\;:\; t \in \m R \} \cong {\mathbb S}^1.$ Then we define ${\rm PU}(n,1) = {\rm U}(n,1) / {\mathbb S}^1$ and note that ${\rm PU}(n,1)$ is the group of holomorphic automorphisms of $\m B^n.$ 

Using \cite[Proposition 7.14]{Helgason} it follows that the compact dual, $X=\mathbb{CP}_n$, of $\m B^n$ can be written as a quotient $G(\mathbb C)/B$ with $G(\mathbb C) = {\rm PGL}_{n+1}(\mathbb C)$ and $B$ the subgroup of classes in ${\rm PGL}_{n+1}(\mathbb C)$ of block triangular matrices of the form
$$
\left( \begin{matrix}
A & 0\\
a_2 & a_0 
\end{matrix} 
\right).
$$
If we let $\Gamma \subset G(\mathbb R)$ be a lattice and let $Y$ be the quotient $\Gamma \backslash \m B^n$, then the quotient map $\upsilon:\m B^n\rightarrow Y$ is a covering map. From \cite[Theorem 3.4.5]{BUball}, we have that $\upsilon$ is definable in $\m R_{an,exp}$ on some fundamental domain. Hence, once again, the uniformizer $\upsilon$ is a solution to a $(G(\mathbb C), G(\mathbb C)/B)$-structure on $Y$  as defined above. Note that by results of Mok \cite{MokB}, the quotient $Y$ has the structure of a quasi-projective algebraic variety. In this setting, Baldi and Ullmo \cite[Theorem 1.22]{BUball} established the Ax-Schanuel conjecture: 

\begin{thm} \label{BUAS} 
Let $W \subset \m B^n \times Y$ and $D$ be the graph of the quotient map. Let $U$ be an irreducible component of $ W \cap D$ such that $\codim U < \codim W + \codim D$ or equivalently $\dim W < \dim U + \dim Y.$ If the projection from $U$ to $Y$ is positive dimensional, then it is contained in a strict totally geodesic subvariety of $Y$. 
\end{thm}
Theorem \ref{BUAS} generalizes the earlier non-arithmetic Ax-Lindemann-Weierstrass theorem of Mok \cite{MokALW}. In the setting of this subsection, by Corollary 5.6.2 of \cite{BUball}, the totally geodesic subvarieties are precisely the bi-algebraic subvarieties for the map $\upsilon.$ Keeping in mind this connection will be essential later for observing applications of our results which generalize the Ax-Schanuel result of \cite{BUball}. There are also similar results of \cite{bader} (stated in terms of geodesic subvarieties) for $\Gamma \leq SO(n,1)$, another setting to which our results would likely apply\footnote{One would simply have to establish that the uniformizers are o-minimally definable.}. 

\subsection{Comparison of special subvarieties} \label{comparisonsection}

Let $\mathscr Y$ be a $(G,G/B)$-structure on an algebraic variety $Y$ and $\upsilon$ a solution of $\mathscr Y$ in $\mathcal O({\rm dom}(\upsilon))$ holomorphic on its domain and $\nabla$ the associated connection. An algebraic subvariety $X \subset Y$ is said to be bi-algebraic if the irreducible components of $\upsilon^{-1}(X)$ are intersection $S \cap {\rm dom}(\upsilon)$ with $S$ algebraic in $G/B$. 

\begin{prop}
 Bi-algebraic subvarieties are $\nabla$-non-generic.
\end{prop}
\begin{proof}
The algebraic subvariety $ S \times X \subset  G/B\times Y$ contains the graph of $\upsilon|_S$, called $U$ and $\dim S\times X = \dim S + \dim X < \dim U + \dim Y$. Corollary \ref{thirdAx} can be applied.
\end{proof}

The converse is, in general, not true. The class of $\nabla$-non-generic varieties is closed under taking subvarieties, while the class of bi-algebraic subvarieties is not. In the setting of Shimura varieties, the class of bi-algebraic varieties are precisely the class of \emph{weakly special subvarieties} (see \ref{weaklyspec} and \ref{weaklyspeccomp}). However, in the setting of mixed period mappings, a general setting including many of the recent results of Ax-Schanuel type, Chiu \cite{Chiu} uses a result of Andr\'e-Deligne to show that every $\nabla$-non-generic subvariety is contained in a proper weakly special subvariety. It follows that in this setting every proper $\nabla$-special subvariety is bi-algebraic. 
 
\section{The product case and applications to Model theory}\label{sectionproduct}
In this section we use Theorem \ref{th:AxSh1} to study products of $(G, G/B)$-structures. In particular we show that an Ax-Schanuel type theorem holds in this setting. We then apply this result to give a model theoretic study of set defined, in a differentially closed field, by the uniformizers of $(G,G/B)$-structures. We show that the definable sets are strongly minimal, geometrically trivial and in the case of covering maps, satisfy a weak form of the Ax-Lindemann-Weierstrass Theorem with derivatives. Before that, we need a technical result on projection of $\nabla$-non-generic varieties.

\begin{lem}\label{lm:projection}
Let $\nabla_1$ and $\nabla_2$ be principal connections on bundles $P_1\to Y_1$ and $P_2\to Y_2$ with connected structure groups $G_1$ and $G_2$ with ${\rm Gal}(\nabla_1) = G_1$ and ${\rm Gal}(\nabla_2) = G_2$. Let $X\subset Y_1\times Y_2$ be a $(\nabla_1\times\nabla_2)$-non-generic subvariety and let us assume ${\rm Gal}((\nabla_1\times\nabla_2)|_X) \subseteq H\times G_2$ for some proper algebraic subgroup $H\subset G_1$. Then there is a $\nabla_1$-non-generic subvariety $X_1\subset Y_1$ such that $X\subseteq X_1\times Y_2$ and $ {\rm Lie}({\rm Gal}(\nabla_1|_{X_1})\subseteq {\rm Lie}(H)$.  
\end{lem}

 \begin{proof}
Let $X_1$ be the Zariski closure of the projection of $X$ in $Y_1$. Let us take any subvariety $X'\subset X$ such that the projection ${\rm p}\colon X'\to X_1$ is a finite-to-one dominant map. By hypothesis 
 $${\rm Gal}((\nabla_1\times\nabla_2)|_{X'}) \subseteq  {\rm Gal}(\nabla_1\times\nabla_2|_{X})\subseteq H\times G_2.$$
 On the other hand, as a connection on $X'$, ${\rm p}^*\nabla_1$ is the quotient of ($\nabla_1\times\nabla_2)|_{X'}$ by the action of $G_2$. It follows that ${\rm Gal}({\rm p}^*\nabla_1)\subseteq H$. As ${\rm p}$ is a finite-to-one map we have ${\rm Gal}({\rm p}^*\nabla_1)$ and ${\rm Gal}(\nabla_1|_{X_1})$ have the same Lie algebra included in ${\rm Lie}(H)$. Hence $X_1$ is $\nabla$-non-generic.
 \end{proof}

\subsection{Product of $(G, G/B)$-structures.}
In this subsection we apply some of the results of Section 3 to products of $(G, G/B)$-structures. We first observe that the following Lemma holds.
\begin{lem}\label{lm:total_galois}
Let $G$ be an algebraic group, $B$ a subgroup and $\mathscr Y \subset J^\ast(G/B, Y)$ a rational $(G, G/B)$-structure on $Y$. If the Galois group of the associated charts set $\mathscr C$ is $G$ then 
\begin{enumerate}
    \item $\mathscr Y$ is irreducible,
    \item there is no proper $\mathcal D_{G/B}$-subvariety of $\mathscr Y$.
\end{enumerate}
\end{lem}

\begin{proof} 
Using the isomorphism  between $J^\ast(G/B, Y)$ and  $J^\ast(Y, G/B)$, we need to prove that $\mathscr C$ is irreducible. The hypothesis on the Galois group means that $\mathscr C$ is the Zariski closure of an horizontal leaf. Irreducibility of the Zariski closure of a leaf of an holomorphic foliation is  proved in \cite{bonnet}.

The isomorphism  between $J^\ast(G/B, Y)$ and  $J^\ast(Y, G/B)$ exchanges the two differential structures thus a  proper $\mathcal D_{G/B}$-subvariety of $\mathscr Y$ gives a  proper $\mathcal D_{Y}$-subvariety of $\mathscr C$. By our Galois assumption there is no $\mathcal D_{Y}$-subvariety of $\mathscr C$. 
\end{proof}

For $i=1 ,\ldots, n$ let $G_i$ be a simple algebraic group, $B_i \subset G_i$ a subgroup and $\mathscr C_i \subset J^\ast(Y_i, G_i/B_i)$ a rational $( G_i, G_i/B_i)$-structure on $Y_i$ whose $G_i$-invariant connection is denoted by $\nabla_i$. 
 As 
 $$\prod_{i=1}^n J^*(Y_i,G_i/B_i) \subset J^*\left(\prod_{i=1}^n Y_i, \prod_{i=1}^n G_i/B_i \right) = \tilde J$$ 
 the subset $\mathscr C = \prod_{i} \mathscr C_i $ is a $(\prod_{i} G_i , \prod_i G_i/B_i)$-structure on $Y = \prod_i Y_i$ whose connection will be $\nabla$. A uniformization of this structure is a map $\bar\upsilon = (\upsilon_1 , \ldots , \upsilon_n) \in \mathscr Y$, where the $i^{th}$ factor is a uniformization of $Y_i$ and depends only of variables $G_i/B_i$.  \\

For each group $G_i$ let us consider $Z_i$ its center, $\bar G_i = G_i/Z_i$ and $\bar B_i = B_i/(B_i\cap Z_i)$. There is a natural quotient map $G_i/B_i \to \bar G_i/\bar B_i$ that is a Galois covering of fiber $Z_i/(B_i\cap Z_i)$. The composition with this covering map induces also a covering map $J^*(Y_i,G_i/B_i) \to J^*(Y_i,\bar G_i/\bar B_i)$ that maps the rational $(G_i,B_i)$-structure $\mathscr C_i$ onto a rational $(\bar G_i,\bar B_i)$ structure $\bar{\mathscr C}_i$ on $Y_i$. As a principal bundle, this structure $\bar{\mathscr C}_i$ is the quotient $\mathscr C_i/Z_i$. The invariant connection on $\bar{\mathscr C}_i$ is denoted $\bar \nabla_i$. In the following theorem, by a correspondence we mean an algebraic subvariety whose projections on both factors are finite and dominant.

\begin{thm}\label{Binary}
 Let $V$ be an algebraic subvariety of $\tilde{J}$ and let $\mathcal L \subset \tilde{J}$ be the graph of a uniformizer in $\mathscr Y$. Assume 
 \begin{enumerate}
     \item The Galois group of the $i^{th}$ factor $\mathscr C_i$ is $G_i$,
     \item $V$ is the Zariski closure of an irreducible component $\mathcal V$  of $V \cap \mathcal L$,
     \item $\dim V < \dim(\mathcal V) + \dim G_1 +\ldots+ \dim G_n$ ,
     \item the projection of {$\mathcal V$} on each $G_i/B_i$ is dominant\footnote{ By dominant map between analytic sets we means that the image contains a Euclidean open subset}.
 \end{enumerate}   
Then
\begin{enumerate}
\item[(a)]there exist two indices $i<j$ such that 
the projection $X_{ij}$ of $V$ in $Y_i \times Y_j$ is a $(\nabla_i\times\nabla_j)$-non-generic subvariety whose projections on both factors is onto.
\item[(b)] 
Denote by $\pi_i$ and $\pi_j$ the projections of $X_{ij}$ on $Y_i$ and $Y_j$. There is an isomorphism 
$$\varphi \colon \pi_i^*\bar{\mathscr C}_i\to \pi_j^*\bar{\mathscr C}_j$$  
defined over $X_{ij}$ such that $\varphi_*\pi_i^*\bar\nabla_i = \pi_j^*\bar\nabla_j$.
 \item[(c)] Assume the factors have finitely many $Z_i$-special subvarieties of positive dimension, then $X_{ij}$ is a correspondence. 

 \end{enumerate}

\end{thm}

\begin{proof}

(a) Let $X$ be the Zariski-closure of the projection of $V$ in $Y_1 \times \ldots \times Y_n$. The group $\prod_i G_i$ is semi-simple thus it is sparse.
By assumption (3) and Theorem \ref{th:AxSh1}, $X$ is a $\nabla$-non-generic subvariety of $Y_1\times\ldots \times Y_n$. The projection $\mathscr C \to \mathscr C_i$ is compatible with the connections. As the projection of $\mathcal V$ on $G_i/B_i$ is dominant, so is the projection of $\mathcal V$ into the graph of a chart in $G_i/B_i \times Y_i$. This implies that the projection $\mathscr C|_X \to \mathscr C_i$ has a $\nabla_i$-invariant image. By lemma  \ref{lm:total_galois}, this projection is dominant. 
If follows that ${\rm Gal}(\nabla|_X)$ is a subgroup of $G_1\times\ldots\times G_n$ that projects onto each component $G_i$. Then, by Goursat's lemma there are indices $i < j$ and an algebraic group isomorphism,
$\sigma\colon \bar G_i \to \bar G_j$
such that
$${\rm Gal}(\nabla|_X) \subset \left\{(g_1,\ldots,g_n)\in \prod_i G_i \;\colon\; \sigma( g_i Z_i) = g_j Z_j\right\}$$
Now, let us consider $\nabla_{ij} = \nabla_i \times \nabla_j$ as a connection on 
$$\mathscr C_i \times \mathscr C_j \to Y_i\times Y_j.$$
In the case $n = 2$ we have that $X = X_{ij}$ is a $\nabla$-non-generic subvariety. If $n>2$ we decompose,
$$\prod_{k=1}^n \mathscr C_k = (\mathscr C_i\times \mathscr C_j)\times \prod_{k\neq i,j} \mathscr C_k,$$ 
then ${\rm Gal}(\nabla|_{X}) \subset G_{\sigma}\times \prod_{k\neq i,j} G_k$ with
$G_\sigma = \{(g_i,g_j)\in G_i\times G_j\colon \sigma( g_i Z_i) \subset  g_j Z_j\}$. 
By Lemma \ref{lm:projection}, we have that the image $X_{i,j}$ of $X$ by the projection on $Y_i\times Y_j$ is a proper $(\nabla_i\times\nabla_j)$-non-generic subvariety.

(b) Let us now consider $\bar \nabla_{ij} = \bar \nabla_i \times \bar \nabla_j$ which is the quotient of $\nabla|_{ij}$ by the action of the finite group $Z_i\times Z_j$. The Galois group of $\bar\nabla_{ij}|_{X_{ij}}$ is included in the image of
$G_\sigma$ by the map
$$G_{\sigma} \to \bar G_i \times \bar G_j, \quad (g_1,\ldots,g_n) \mapsto (g_iZ_i,g_jZ_j).$$
By the above arguments we already know that the Galois group of $\bar\nabla_{ij}|_{X_{ij}}$ is included in
$$\bar G_\sigma = \{ (g_iZ_i,g_jZ_i)\in \bar G_i\times \bar G_j\;\colon\; \sigma(g_iZ_i) = g_jZ_j \}.$$
Moreover, by Goursat's lemma again, this group is also minimal among subgroups of $\bar G_i\times \bar G_j$ that projects onto the two factors. Therefore, the Galois group of  $\bar\nabla_{ij}|_{X_{ij}}$ is exactly $\bar G_{\sigma}$.
Let $T_{ij}$ be the Zariski closure of an horizontal leaf of $\bar \nabla_{ij}|_{X_{ij}}$ with Galois group $\bar G_{\sigma}$ then $T_{ij}\subset \bar{\mathscr C}_i\times \bar{\mathscr C}_j$ is the graph of an isomorphism of principal bundle with connections.

(c) We may consider $Y_i^\circ \subset Y_i$ the complement of the finitely many $Z_i$-special subvarieties of positive dimension. One can restrict bundles and connections above the products of $Y_i^\circ$. Thus we may assume that factors have no $Z_i$-special subvarieties of positive dimension.

As $X_{ij} \subset Y_i \times Y_j$ is a $G_\sigma$-non-generic subvariety and for $y \in Y_i$, $\{y\}\times Y_j$ is a $(\{e\} \times G_j)$-non-generic subvariety, then the intersection  $X_j = X_{ij} \cap \{y\}\times Y_j$ is empty, has dimension $0$ or is a positive dimensional non generic subvariety with group $G_\sigma \cap (\{e\} \times G_j)$. As $G_\sigma$ is the graph of an isomorphism up to centers the latter is $\{e\} \times Z_j$. As $Y_j$ has no $Z_j$-special subvarieties of positive dimension, $X_j$ is empty or has dimension $0$.

The projections of $X_{ij}$ are onto, thus $X_j$ is generically (on $y$) not empty.  Hence $\dim X_{ij} = \dim Y_i$ and $\dim X_{ij} =\dim Y_j$ and the two projections are dominant. This proves the assertion (c), that is $X_{ij}$ is a correspondence.
\end{proof}

In the following definition, we extract some of the key properties of geometric structures needed to apply the above theorem and to formulate a number of applications

\begin{defn}\label{simple}
A $(G, G/B)$ structure  $\mathscr Y$ (or $\mathscr C$) on an algebraic variety $Y$ is said to be {\em simple} if
\begin{enumerate}
    \item $G$ is a simple group with finite center $Z$,
    \item The Galois group of $\mathscr C$ is $G$,
    \item $Y$ has finitely many irreductible $Z$-special subvarieties of positive dimension.
\end{enumerate}
\end{defn}

\begin{rem}
If $\dim Y = 1$ the condition (3) trivially holds since any curve with a geometric structure has no special subvarieties of positive dimension.
\end{rem}

\begin{exam} 
 Consider again Example \ref{Appell} in the case when the solutions of $\mathscr Y_{hyp}$ are built from the quotient of the ball $\mathbb B \subset \mathbb{CP}_2$ by a lattice $\Gamma \subset {\rm PSU}(2,1) \subset {\rm PSL_3(\mathbb C)}$. As this lattice is included in ${\rm Gal}(\mathscr C_{hyp})$, the Galois group of this geometric structure is ${\rm PSL}_3(\mathbb C)$ and it satisfies (1) and (2) in definition \ref{simple}.
 
 To see that the third condition is satisfied, let us consider $X$ a $Z$-special curve in $Y$ and denote $\upsilon : \mathbb B \to Y$ the quotient map. The condition ${\rm Gal(\mathscr C_{hyp}|_X)} = Z$ implies that the restriction of an inverse branch $\upsilon^{-1}|_X$ is an algebraic map on $X$ with values in $\mathbb B$. By Liouville's theorem $\upsilon^{-1}|_X$ must be constant which is in contradiction with the definition of $\mathscr C_{hyp}$. 
\end{exam}

The above example is simply a special case of a more general phenomenon: 
\begin{exam}Let $\upsilon : \Omega \to Y=\Gamma \backslash \Omega$ be the quotient of a bounded symmetric domain $\Omega$ with automorphisms group $G(\mathbb R)$, $G$ simple and $\Gamma \subset G$ a lattice.  These contain the case of the uniformization of a Shimura variety with simple group $G$. 

As we have seen in Section \ref{coveringMap}, the covering map $\upsilon$ is a uniformization of a $(G,G/B)$-structure and we denote the associated connection by $\nabla$. By construction we have that $\Gamma \subset {\rm Gal}(\nabla)$ and hence it follows that ${\rm Gal}(\nabla) = G$. 

It is not hard to see that the third condition is also satisfied. If we assume that $X \subset Y$ is $Z$-special, then the restriction of the chart $\upsilon^{-1}$ to $X$ is an algebraic function on $X$ with values in $\Omega$. But there are no bounded algebraic functions on $X$ unless if $\dim X =0$. So Theorem \ref{thmC} from the Introduction is now obtained as a corollary to Theorem \ref{Binary}
\end{exam}

\begin{cor}[\bf Theorem C] \label{BinaryTransc}
Let $(Y, \mathscr Y)$ be a simple $(G, G/B)$-structure on $Y$, $\hat p_1 \ldots, \hat p_n$ be $n$ formal parametrizations of (formal) neighborhoods of points $p_1, \ldots, p_n$ in $G/B$ and $\upsilon_1, \ldots , \upsilon_n$ be solutions of $\mathscr Y$ defined in a neighborhood of $p_1, \ldots, p_n$ respectively.
If $$
{\rm tr.deg.}_{\mathbb C} \mathbb C\left( \hat p_i , (\partial^\alpha \upsilon_i)(\hat p_i) : 1 \leq i \leq n,\; \alpha \in \mathbb N^{\dim Y} \right) < \dim Y + n\dim G$$
then there exist $i<j$ such that 
$$
{\rm tr.deg.}_{\mathbb C} \mathbb C (\upsilon_i(\hat p_i), \upsilon_j(\hat p_j)) = {\rm tr.deg.}_{\mathbb C} \mathbb C(\upsilon_i(\hat p_i) ) = {\rm tr.deg.}_{\mathbb C} \mathbb C(\upsilon_j(\hat p_j)) = \dim G/B.$$
\end{cor}

\begin{proof}
For $\upsilon : U \to Y$ a solution of $\mathscr Y$, one denote by $j\upsilon : U \to J^\ast(U , Y)$ the map obtained from the Taylor expansions. If $\upsilon$ is a solution of $\mathscr Y$ then $j\upsilon$ takes values in $\mathscr Y\subset J^\ast(G/B, Y)$
As $v_i(\hat p_i)$ and $v_i(\hat p_j)$ are Zariski dense in $G/B$, we obtain that the transcendence degree of the field spanned by any of them is equal to {$\dim G/B=\dim Y$}.

As explained in the discussion before Theorem \ref{Binary}, the product $\mathscr Y^n \to Y^n$ is a $(G^n, (G/B)^n )$ structure on $Y^n$. Let $V$ the Zariski
closure of $(j\upsilon_1(\hat p_1), \ldots, j\upsilon_n(\hat p_n))$, $\mathcal L$ the horizontal leaf that passes through the point
$(j\upsilon_1(p_1), \ldots, j\upsilon_n(p_n))$,
and $\mathcal V$ the component of the intersection $\mathcal L\cap V$
containing the point $(j\upsilon_1(p_1), \ldots, j\upsilon_n(p_n))$. By construction we have that $
{\rm tr.deg.}_{\mathbb C} \mathbb C\left( \hat p_i , (\partial^\alpha \upsilon_i)(\hat p_i) : 1 \leq i \leq n,\; \alpha \in \mathbb N^{\dim Y} \right)$ is equal to the complex dimension of $V$. Hence the hypothesis of Theorem \ref{Binary} is satisfied and we obtain a correspondence $X_{ij}\subset Y^2$ between $\upsilon_i(\hat p_i)$ and $\upsilon_j(\hat p_j)$ as required.
\end{proof}

\subsection{Strong minimality of the differential equations for uniformizers.} \label{Urk} We begin by recalling some of the relevant notions from the model theoretic approach to the study of differential equations. We let ${\mathcal L}_{m}=\{0,1,+,\cdot\}\cup \Delta$ denote the language of differential rings, where $\Delta=\{\partial_1,\ldots,\partial_m\}$ is a set of unary function symbols. From a model theoretic perspective, differential fields are regarded as ${\mathcal L}_{m}$-structures where the symbols $\partial_i$ are interpreted as commuting derivations, while the other symbols are interpreted as the usual field operations. 

A differential field $(K,\Delta)$ is differentially closed if it is existentially closed in the sense of model theory, namely if every finite system of $\Delta$-polynomial equations with a solution in a $\Delta$-field extension already has a solution in $K$. We use m-$DCF_0$ to denote the common first order\footnote{The description  given here is not a first order axiomatization. We refer the reader to \cite{McGrail} for the basic model theory of m-$DCF_0$.} theory of differentially closed fields in ${\mathcal L}_{m}$. It follows that m-$DCF_0$ has quantifier elimination, meaning that every definable subset of a differentially closed field $(K,\Delta)$ is a Boolean combination of Kolchin closed sets.

\begin{rem}
We will use the following model theoretic conventions
\begin{enumerate}
    \item The notation $\upsilon$ will be used both for a tuple and an element.
    \item We say that a tuple $\upsilon$ is algebraic over a differential field $K$, and write $\upsilon\in K^{\rm alg}$, if each coordinate of $\upsilon$ is algebraic over $K$.
\end{enumerate}
\end{rem}

{We fix a sufficiently saturated model $(\mathbb U,\Delta)$ of m-$DCF_0$ and assume that $\mathbb{C}$, the field of complex numbers, is contained in its field of constants, i.e., $\mathbb{C}\subset\{\upsilon\in\mathbb{U}:\partial(\upsilon)=0\text{ for all }\partial\in\Delta\}$.} Given a differential subfield $K$ of $\mathbb U$ and $\upsilon$ a tuple of elements from $\mathbb{U}$, the complete type of $\upsilon$ over $K$, denoted $tp(\upsilon/K)$, is the set of all ${\mathcal L}_{m}$-formulas with parameters in $K$ that $\upsilon$ satisfies. It is not hard to see that the set $$I_{p, K}=\{f\in K\{X\}: (f(X)=0)\in p\}=\{f\in K\{X\}: f(\upsilon)=0\}$$ is a differential prime ideal in the differential polynomial ring $K\{X\}$, where $p=tp(\upsilon/K)$. Using quantifier elimination, it is not hard to see that the map $p\mapsto I_{p, K}$ is a bijection between the set of complete types over $K$ and differential prime ideals in $K\{X\}$. Furthermore, it follows that a tuple $\upsilon_1$ is a realization of $tp(\upsilon/K)$ if and only if $I_{p, K}$ is the vanishing ideal of $\upsilon_1$ over $K$. Therefore in what follows there is no harm to think of $p=tp(\upsilon/K)$ as the ideal $I_{p, K}$.

\begin{defn} 
Let $K\subset\mathbb U$ be a differential field and $\upsilon$ be a tuple from $\mathbb U$.
\begin{enumerate}
\item Let $F\subset\mathbb U$ be a differential field extension of $K$. We say that $tp(\upsilon/F)$ {\em is a nonforking extension} of $tp(\upsilon/K)$ if $K\gen{\upsilon}$ is algebraically disjoint from $F$ over $K$, i.e., if $y_1,\ldots,y_k\in K\gen{\upsilon}$ are algebraically independent over $K$ then they are algebraically independent over $F$. We say that $tp(\upsilon/F)$ is a forking extension of $tp(\upsilon/K)$ if it is not a nonforking extension.
\item We say that $tp(\upsilon/K)$ has {\em $U$-rank 1} (or is minimal) if and only if $\upsilon\not\in K^{\rm alg}$ but every forking extension of $tp(\upsilon/K)$ is algebraic, that is has only finitely many realizations.
\end{enumerate}
\end{defn}

\begin{rem} Let $K$ and $\upsilon$ be as above and let $p=tp(\upsilon/K)$. Let $F$ be a differential field extension. Assume further that ${\rm tr.deg.}_KK\gen{\upsilon}=r$.
\begin{enumerate}
\item If $K$ is algebraically closed then $p$ has a unique non-forking extension to $F$, namely $tp(\hat\upsilon/F)$ for any $\hat\upsilon$ realizing $p$ such that ${\rm tr.deg.}_FF\gen{\hat\upsilon}=r$.
\item We have that $tp(\upsilon/F)$ is algebraic if and only if  $\upsilon\in F^{\rm alg}$.
\item Also $tp(\upsilon/F)$ is a nonforking extension of $tp(\upsilon/K)$ if and only if ${\rm tr.deg.}_KK\gen{\upsilon}={\rm tr.deg.}_FF\gen{\upsilon}$. 
\item In particular, the assumptions $p$ has $U$-rank 1 and ${\rm tr.deg.}_FF\gen{\upsilon}<r$ (so that $tp(\upsilon/F)$ is a forking extension of $p$) implies that $\upsilon\in F^{\rm alg}$.
\end{enumerate}
\end{rem}

Let $\mathscr Y\subset \mathbb{U}^\ell$ be a definable set and $K$ any differential field over which $\mathscr Y$ is defined. Assume that the order ${\rm ord}(\mathscr Y)=\sup\{{\rm tr.deg.}_KK\gen{\upsilon}: \upsilon\in \mathscr Y\}$ is finite; say ${\rm ord}(\mathscr Y)=r$. By the (complete) {\em type $p$ of $\mathscr Y$ over $K$} we mean that $p=tp(\upsilon/K)$ for any $\upsilon\in \mathscr Y$ such that ${\rm tr.deg.}_KK\gen{\upsilon}=r$.  Recall that we say that $\mathscr Y$ is {\em strongly minimal} if it cannot be written as the disjoint union of definable sets of order $r$, and for {\em any} differential field extension $F$ of $K$ and element $\upsilon\in \mathscr Y$, we have that ${\rm tr.deg.}_F F\gen{\upsilon}=0\text{ or }r$. We will make use of the following fact.
\begin{fct}\label{smU-rank1}
The definable set $\mathscr Y$ is strongly minimal if and only if its type over $K$ has U-rank 1, if {$\mathscr Y$ cannot be written as the disjoint union of $K^{\rm alg}$-definable sets of order $r$}, and for {\em any} element $\upsilon\in \mathscr Y$, we have that ${\rm tr.deg.}_KK\gen{\upsilon}=0\text{ or }r$.\footnote{This formulation is precisely suited for the argument we give later in the paper, and we have phrased it in this way so that it might be used more easily as a black box for non-experts. Taken together, the conditions might equivalently be written as - $\mathscr Y$ is the zero set of a prime differential ideal $P$, such that any differential ideals containing $P$ (even after base change to a larger differential field) have the property that their zero sets are finite.}  
\end{fct}

\begin{defn}
Let $\mathscr{Y}\subset \mathbb{U}^m$ be a strongly minimal set and $K$ any differential field over which $\mathscr{Y}$ is defined. We say that $\mathscr{Y}$ {\em geometrically trivial} if for any distinct $\upsilon,\upsilon_{1},\ldots,\upsilon_{\ell}\in\mathscr{Y}$, if $\upsilon\in K\gen{\upsilon_{1},\ldots,\upsilon_{\ell}}^{\rm alg}$ then for some $1\leq i\leq\ell$, we have that $\upsilon\in K\gen{\upsilon_{i}}^{\rm alg}$.
\end{defn}

\sloppy We now aim to show that the set defined by the partial differential equations for any uniformizer is strongly minimal and geometrically trivial. We will need a general tool/phenomenon from model theory (more precisely stability theory) which has sometimes been informally called the {\em Shelah reflection principle} (what we need is a well-known consequence of definability of types in stable theories). We restrict our exposition to types of finite order. Let $F=F^{\rm alg}$ be any algebraically closed differential field and let $p=tp(\upsilon/ F)$ for some tuple $\upsilon$. Assume that ${\rm tr.deg.}_FF\gen{\upsilon}=\ell$. We say that a sequence $(\upsilon_i)_{i=1}^{\infty}$ is a {\em Morley sequence in $p$} if $\upsilon_{i+1}$ realizes the (unique) \emph{non forking extension} of $p$ over $F_i=F\gen{\upsilon_1,\ldots,\upsilon_i}^{\rm alg}$, i.e., in particular ${\rm tr.deg.}_{F_i}F_i\gen{\bar \upsilon_{i+1}}=\ell$. It follows that one can take $\upsilon_1=\upsilon$.

In general, when given a differential variety $\mathscr Y $ defined over a differential field $K$, $p$ the type of a generic solution of $\mathscr Y$ over $K$, $\upsilon $ a realization of $p$, and $F$ a differential field extension of $K$, we say that $tp(\upsilon / F )$ is a \emph{forking extension} of $p$ if ${\rm tr.deg.}_F F\gen{\upsilon} < {\rm tr.deg.}_K K\gen{\upsilon}.$ Otherwise, $tp(\upsilon / F )$ is a non forking extension of $p$ (if $K$ is algebraically closed, then this extension is unique). The next result gives a characterization of the kinds of fields one needs to consider while characterizing forking extensions of a type. 

\begin{fct}\label{NotSelah}{\cite[Lemma 2.28]{GST}}
Let $K$ be any differential field and let $p=tp(\upsilon/K)$ for some tuple $\upsilon$ with ${\rm tr.deg.}_KK\gen{\upsilon}=r$. Let $F=F^{\rm alg}$ be an algebraically closed differential field extension of $K$ such that $tp(\upsilon/F)$ is a forking extension of $p$. Then there is a finite initial segment $(\upsilon_1,\ldots,\upsilon_k)$ of a Morley sequence $(\upsilon_i)_{i=1}^{\infty}$ in $tp(\upsilon/F)$ such that ${\rm tr.deg.}_{K}K\gen{\upsilon_1,\ldots,\upsilon_k}<k\cdot r$. Note that the hypotheses here imply (assuming one selects a sequence with the given properties of minimal length) that ${\rm tr.deg.}_{F}F\gen{\upsilon_k} \leq {\rm tr.deg.}_{K\gen{\upsilon_1,\ldots,\upsilon_{k-1}}}K\gen{\upsilon_k}<r$, while $(\upsilon _1 , \ldots , \upsilon _{k-1} )$ is a Morley sequence over $K$ (and thus that ${\rm tr.deg.}_{K}K\gen{\upsilon_1,\ldots,\upsilon_{k-1}}=(k-1)\cdot r$). \footnote{In fact, one can actually arrange that ${\rm tr.deg.}_{F}F\gen{\upsilon_k} = {\rm tr.deg.}_{K\gen{\upsilon_1,\ldots,\upsilon_{k-1}}}K\gen{\upsilon_k}$ or even more specifically that a \emph{canonical base} for the forking extension is contained in the algebraic closure of the initial segment of the Morley sequence.}
\end{fct}
The intuition here is that we start with a Morley sequence in $tp(\upsilon/F)$ and forking is captured in the fact that at some point this appropriately chosen sequence ceases to be a Morley sequence in $p=tp(\upsilon/K)$ (recall that we can start with $\upsilon_1=\upsilon$). We can now apply Theorem \ref{Binary} to study the set defined by the differential equations for uniformizers. Let us first explain the translation from the geometrical setting to m-$DCF_0$. 

We have fixed a simple $(G, G/B)$-structure $\mathscr C$ on a dimension $m$ algebraic variety $Y$  with $\dim(G)=k$. We consider an open subset of $U$ of $G/B$ and assume that $\bar t =(t_1, \ldots, t_m)$ are the coordinates on $U$ realizing a transcendence basis of $\mathbb C(G/B)$. We assume $Y \subset \mathbb C^\ell$. Now  a uniformization of the $(G,G/B)$-structure on $Y$, say $\upsilon :U \to Y \subset \mathbb C^\ell$, will be described by a system of partial differential equations in variables $t_1, \ldots ,t_m$ and unknowns the $\ell$ coordinates of $\upsilon$. 

{We assume throughout that our universal field $\mathbb{U}$ contains elements $t_1, \ldots, t_m$ such that $\partial_i t_i=1$ and $\partial_j t_i=0$. We denote by $\mathscr Y$ the $(\partial_1,\ldots,\partial_m)$-differential equations satisfied by $\upsilon$ together with the inequations ensuring that the rank of the Jacobian matrix of solutions is $m$. By abuse of notation, $\mathscr Y\subset \mathbb{U}^{\ell}$ also denotes the solution set it defines. Next let $K$, with $\mathbb{C}\subseteq K\subseteq \mathbb{C}(\bar t)^{\rm alg}$, for some (any) differential field over which $\mathscr Y$ is defined. In general $K\neq\mathbb{C}$.

\begin{rem}\label{genericAll} We have the following observations
 \begin{enumerate}
 \item The assumption on the rank of the Jacobian matrix is implicitly part of the formalism of Subsection \ref{sectionUniformization}. When $\dim(Y)=1$, this corresponds to the assumption that one only considers non-constant solutions.
 \item Using Lemma \ref{lm:total_galois} (2) for any $\upsilon\in \mathscr Y$, we have that ${\rm tr.deg.}_{\mathbb{C}(\bar t)}\mathbb{C}(\bar t)\gen{\upsilon}=k$. Hence it also follows that ${\rm tr.deg.}_{K}K\gen{\upsilon}=k$. 
 \end{enumerate}
\end{rem}
}

The following statement, which is needed to prove the next theorem, is an abstract reformulation of a special case of Corollary \ref{BinaryTransc}. 

\begin{cor}\label{CorMainHigher}
Let $\mathscr Y$ be a $(G,G/B)$-structure on an $m$-dimensional algebraic variety $Y$ with simple Galois group $G$ of dimension $\dim(G)=k$. Let $\upsilon_1,\ldots,\upsilon_n \in \mathscr Y$ be distinct solutions. If

$$\text{\rm tr.deg.}_{\mathbb{C}(\bar t)}{\mathbb{C}(\bar t)}\gen{\upsilon_1,\ldots,\upsilon_n}<kn,$$
then for some $i<j$, we have that 
$${\rm tr.deg.}_{\mathbb C(\upsilon_i)}\mathbb{C}(\upsilon_i,\upsilon_j) \leq m-1.$$
If we further assume that $(Y,\mathscr Y)$ has finitely many maximal positive dimensional $Z$-special subvarieties, where $Z$ is the center of $G$, then $\upsilon_i\in \mathbb{C}(\upsilon_j)^{\rm alg}$, that is each component of $\upsilon_i$ is algebraic over $\mathbb{C}(\upsilon_j)$.
\end{cor}

\begin{proof}
Assume that $\text{tr.deg.}_{\mathbb{C}(\bar t)}{\mathbb{C}(\bar t)}\gen{\upsilon_1,\ldots,\upsilon_n}=r<kn$ and let $\mathcal{C}$ be a finitely generated (over $\mathbb{Q}$) algebraic closed subfield of $\mathbb{C}$ such that $$\text{tr.deg.}_{\mathcal{C}(\bar t)}{\mathcal{C}(\bar t)}\gen{\upsilon_1,\ldots,\upsilon_n}=r.$$ For example we can take $\mathcal{C}$ to be generated by the complex coefficients of the polynomial defining the algebraic relations over $\mathbb{C}(\bar t)$ between $\upsilon_1,\ldots,\upsilon_{n}$ and derivatives.

Applying Seidenberg's embedding theorem to the field ${\mathcal{C}(\bar t)}\gen{\upsilon_1,\ldots,\upsilon_{n}}$, we may assume that $\upsilon_1,\ldots,\upsilon_{n}$ are elements of $\mathcal{M}(U)$, the field of meromorphic functions on an open connected domain $U\subset \mathbb{C}^m$. Using Theorem \ref{Binary}, since $$\text{tr.deg.}_{\mathcal C}{\mathcal C(\bar t)}\gen{\upsilon_1,\ldots,\upsilon_{n}}<kn+m,$$ for some $1\leq i <j\leq n$, we have that ${\rm tr.deg.}_{\mathbb C(\upsilon_i)}\mathbb{C}(\upsilon_i,\upsilon_j) \leq m-1$. 

If we further assume that $(Y,\mathscr Y)$ has finitely many maximal irreducible $Z$-special subvarieties then from Theorem \ref{Binary} (c) we get that $\upsilon_i\in \mathbb{C}(\upsilon_j)^{\rm alg}$.
\end{proof}

\begin{thm}\label{Urank1}
Assume that $(Y,\mathscr{Y})$ is simple. Then $\mathscr{Y}$ is strongly minimal and  geometrically trivial. Furthermore, if we let $(F,\partial_1,\ldots,\partial_m)$ be a differential extension of $K$ and let $\upsilon_1$, $\upsilon_2\in \mathscr Y$ with $\upsilon_1$, $\upsilon_2\not\in \mathscr Y(F^{\rm alg})$ then if $\upsilon_2\in F\gen{\upsilon_1}^{\rm alg}$, we have that $\upsilon_2\in \mathbb{C}(\upsilon_1)^{\rm alg}$.
\end{thm}

\begin{proof}
Using Lemma \ref{lm:total_galois} and Fact \ref{smU-rank1} (see also remark \ref{genericAll}) to show that $\mathscr{Y}$ is strongly minimal, all we have to show is that its type over $K$ has $U$-rank 1. Let $\upsilon \in \mathscr Y$ be such that $p=tp(\upsilon/K)$ is the type of $\mathscr Y$ over
$K$. {As pointed out in Remark \ref{genericAll} we have that $\text{\rm tr.deg.}_{K}K\gen{\upsilon}=k$.}
We need to show that every forking extension of $p$ is algebraic. Suppose that $F=F^{\rm alg}$ is a differential field extension of $K$ such that $q=tp(\upsilon/F)$ is a forking extension of $p$. Using Fact \ref{NotSelah} we can hence find distinct $\upsilon_1,\ldots,\upsilon_{r+1}\in \mathscr Y$, an initial segment in a Morley sequence in $q$, such that
\begin{itemize}
    \item $\text{tr.deg.}_{K}{K}\gen{\upsilon_1,\ldots,\upsilon_r}=k\cdot r$; but
    \item $\text{tr.deg.}_{K}{K}\gen{\upsilon_1,\ldots,\upsilon_{r+1}}<k\cdot (r+1)$.
\end{itemize}
Using Corollary \ref{CorMainHigher}, since $\text{tr.deg.}_{\mathbb C(\bar t)}{\mathbb C(\bar t)}\gen{\upsilon_1,\ldots,\upsilon_{r+1}}<k\cdot (r+1),$ for some $1\leq i\leq r$, we have that $\upsilon_{r+1}\in \mathbb{C}(\upsilon_i)^{\rm alg}$. Note that we use here that $\upsilon_1,\ldots,\upsilon_r$, and derivatives are algebraically independent over $K$. Hence $\text{tr.deg.}_{K}{K}\gen{\upsilon_1,\ldots,\upsilon_{r+1}}=k\cdot r$, that is $\upsilon_{r+1}\in {K}\gen{\upsilon_1,\ldots,\upsilon_r}^{\rm alg}$. It hence follows that the only way forking can occur is if $q$ is algebraic which is what we aimed to show.

Finally, given that $\mathscr Y$ is strongly minimal, the statement of Corollary \ref{CorMainHigher} is precisely geometric triviality. The last implication is a direct consequence of Corollary \ref{CorMainHigher} and geometric triviality.
\end{proof}

{\subsection{The case of covering maps}
Assume now that $\upsilon :U \to Y$ is a covering map. So we have $\Gamma $ a Zariski dense subgroup of $G (\m C)$ with the property that the induced action of $\Gamma$ on $G/B$ preserves $U$ and $\upsilon$ is a covering map of the complex algebraic variety $Y$ expressing $Y(\m C)$ as $\Gamma \setminus U$. We also assume that $\upsilon$ satisfies uniformizing algebraic differential equations. For example as in Section 4, this follows if  we assume, in addition, that the restriction of $\upsilon$ to some set containing a fundamental domain is definable in an o-minimal expansion of the reals as an ordered field.

Let $\text{Comm}_G(\Gamma)$ be the commensurator of $\Gamma$. Recall that by a $\text{Comm}_G(\Gamma)$-correspondence (also known as Hecke correspondence) on ${Y}(\mathbb{C})\times {Y}(\mathbb{C})$ we mean a subset of the form $$X_g=\{\upsilon(\bar t)\times \upsilon(g\cdot\bar t):\bar t \in U\}$$
where $g\in \text{Comm}_G(\Gamma)$. It follows that $X_g$ is given by equations $\Phi_{g}(X,Y)=0$ for some set $\Phi_{g}$ of polynomials with complex coefficients. So $\Phi_{g}(\upsilon(\bar t),\upsilon(g\bar t))=0$. With this notation, for $g_{1},g_{2}\in G(\mathbb{C)}$ we more generally say that $\upsilon(g_{1}\bar t)$ and $\upsilon(g_{2}\bar t)$ are in $\text{Comm}_G(\Gamma)$-correspondence if $\Phi_{g}(\upsilon(g_{1}\bar t),\upsilon(g_{2}\bar t))=0$ for some $g\in \text{Comm}_G(\Gamma)$. In other words if $g_1$ and $g_2$ are in the same coset of $\text{Comm}_G(\Gamma)$.

We keep the assumptions and notations from the previous subsection: we assume that our universal differential field $\mathbb{U}$ contains elements $t_1, \ldots, t_m$ such that $\partial_i t_i=1$ and $\partial_j t_i=0$. By abuse of notation, we denote by $\mathscr Y$ the set of solution of the  $(G,G/B)$ structure for $\upsilon$ and we write $K$, with $\mathbb{C}\subseteq K\subseteq \mathbb{C}(\bar t)^{\rm alg}$, for some (any) finitely generated differential field extension of $\m C$ over which $\mathscr Y$ is defined. We also write $k=\dim G$.

\begin{prop}
Let $\upsilon_1,\upsilon_2\in\mathscr Y$ be two distinct solutions. There is an embedding of $K \langle \upsilon_1, \upsilon_2 \rangle $ into the field of meromorphic functions on some open connected domain ${V} \subset G/B$ contained in the fundamental domain of $\Gamma$ such that $\upsilon_i=\upsilon(g_i\bar t)$ for some $g_i\in G(\m C)$. Consequently, if
$\upsilon_1 \in \mathbb C(\upsilon_2)^{\rm alg}$, then $(\upsilon_1, \upsilon_2) \in X_g$ for some $g \in {\rm Comm}_G(\Gamma)$.
\end{prop}
\begin{proof}
We first show that we can write $\upsilon_i=\upsilon(g_i\bar t)$ for some $g_i\in G(\m C)$, where $\upsilon:U\rightarrow Y$ is the covering map. Let $V\subset U$ be a open connected domain which is properly contained in a fundamental domain of action of $\Gamma$. Applying Seidenberg's embedding theorem, we may assume that $\upsilon_1,\upsilon_2$ have coordinates in $\mathcal{M}(V)$, the field of meromorphic functions on $V$. It follows that for some functions $\phi_i:V\rightarrow U$, we can write $\upsilon_i=\upsilon(\phi_i(\bar t))$. Now since $\upsilon_i$ is a solution to $\chi(y)=\tilde\chi(\bar t),$ we have that $$\tilde\chi(\phi_i(\bar t))=\chi(\upsilon(\phi_i(\bar t)))=\tilde\chi(\bar t).$$ From $\tilde\chi(\phi_i(\bar t))=\tilde\chi(\bar t)$ it follows that $\phi_i(\bar t)=g_i\bar t$ for some $g_i\in G(\m C)$.

Finally if we assume there is the set $P$ of polynomials in $\m C[X,Y]$ such that $P(\upsilon_1,\upsilon_2)=0$, then we have that $P(\upsilon(g_1\bar t),\upsilon(g_2\bar t))=0$. Standard arguments using double cosets and the commensurator $\text{Comm}_G(\Gamma)$ of $\Gamma$ (cf.  proof of \cite[Lemmas 5.15]{CasFreNag}) shows that $g_1$ and $g_2$ are in the same coset of $\text{Comm}_G(\Gamma)$. Hence the result follows.
\end{proof}
Combining the above with Theorem \ref{Urank1} we obtain what can be considered a weak form of the Ax-Lindemann-Weierstrass Theorem with derivatives:
\begin{cor}\label{independence}
Assume that $(Y,\mathscr{Y})$ is simple. Let $\upsilon_1,\ldots,\upsilon_n\in \mathscr{Y}$ be distinct solutions that are not in any ${\rm Comm}_G(\Gamma)$-correspondence. Then
$${\rm tr.deg.}_KK\gen{\upsilon_1,\ldots,\upsilon_n}=nk,$$
that is the solutions and their derivatives are algebraically independent over $K$.
\end{cor}

We can use Corollary \ref{independence} and arithmeticity to give a characterization of $\omega$-categoricity of the pregeometry associated with the solution set $\mathscr Y$ generalizing that given in \cite{CasFreNag}. First we recall the following deep result of Margulis (cf. \cite{Marg} or \cite[Chapter 6]{Zimmer}):

\begin{fct} Let $\Gamma$ an irreducible lattice. Then $\Gamma$ is arithmetic if and only if $\Gamma$ has infinite index in ${\rm Comm}_G(\Gamma)$.
\end{fct}

\begin{cor}
Assume that $\Gamma$ an irreducible lattice and $(Y,\mathscr{Y})$ simple. Then $\mathscr{Y}$ is non-$\omega$-categorical if and only if $\Gamma$ is arithmetic.
\end{cor}

Up to commensurability, there are very few known non-arithmetic lattices. For example, we direct the reader to \cite[Section 2.5]{BUball} for a very good summary of the known commensurability classes of non-arithmetic complex hyperbolic lattices.

 \section{Application to products of curves }\label{CurveCase}
In this section, we pay particular attention to the case $\dim Y=1$. We aim to show how all the concepts define in Section \ref{sectionproduct} and Subsection \ref{sectionUniformization} can be explicitly derived in this situation. 
\subsection{Projective structure on curves}

Consider the group $G={\rm PSL}_2(\m C)$ and its subgroup $B$ of lower triangular matrices so that $ G/B=\mathbb{CP}_1$. A $(G, G/B)$-structure on a curve $Y$ is usually called a projective structure (cf.\cite{Dumas}). Let us describe it in the formalism of Subsection \ref{sectionUniformization}. We consider $Y$ a complex affine algebraic curve defined by a polynomial equation,
$$P(y,w) = 0.$$
Without loss of generality we assume that $y$ is a local coordinate at every point of $Y$, that is, the differential form $dy$ has no zeroes on $Y$. The algebraic structure of the jet space $J(Y,\mathbb{CP}_1)$ is given by its ring of regular functions and the latter is the $\mathcal D_Y$-algebra generated by $\mathcal O_Y \otimes \mathcal O_{\mathbb{CP}_1}$. Let us take $t$ to be the affine coordinate in $\mathbb{CP}_1$ and write the ring of regular functions on $J(Y,\mathbb{C}) \subset J(Y,\mathbb{CP}_1)$ as 
 $$
 \mathbb C[Y][t, \dot t, \ddot t, \ldots].
 $$
The $\mathcal D_Y$-differential structure of this ring is given by the action of $$ \frac{d}{dy} =\frac{\partial}{\partial y} + \dot{t}\frac{\partial}{\partial t} + \ddot{t}\frac{\partial}{\partial \dot t} +\dddot{t}\frac{\partial}{\partial \ddot t} + \cdots. $$
The open subset $J^\ast(Y,\mathbb{CP}_1)$ is the set of jets of submersive maps. It is defined by the inequation $\dot t \not = 0$. Its ring of regular functions is denoted by $\mathcal O_{J^\ast}$. To describe the set $\mathscr C \subset J^\ast(Y,\mathbb{CP}_1)$ of jets of charts of the projective structure, we need to introduce the Schwarzian derivative with respect to the coordinate $y$, namely 
$S_y(t) = \frac{\dddot t}{\dot t} -\frac{3}{2} \left( \frac{\ddot t}{\dot t}\right)^2.$

Let $R$ be a function on $Y$ and consider in $\mathcal O_{J^\ast}$ the $\frac{d}{dy}$-ideal generated by $$S_y(t) - 2R(y,w).$$ Since $S_y(t_1) =S_y(t_2)$ if and only if $t_1 = \frac{at_2 +b}{ct_2+d}$ for some $\begin{pmatrix}a&b\\c&d\end{pmatrix}\in {\rm GL}_2(\m C)$, the local analytic solutions of this equation are charts of a projective structure. The zero set of this differential ideal is $\mathscr C$. 
 
Let us give a down to earth description of $\mathscr C$ and the induced $\mathscr Y$. As the equation has order $3$ and is degree 1 in $\dddot t$, $\mathscr C$ is isomorphic as an algebraic variety to $J^*_2(Y,\mathbb{CP}_1)$ and the $\mathcal D_Y$-structure induced  on its ring of regular maps $
 \mathbb C[Y][t, \dot t, \frac{1}{\dot t},\ddot t] $ is given by the above equation
 $$\frac{d}{dy} = \frac{\partial }{\partial y} + \dot t\frac {\partial}{\partial t} + \ddot t\frac {\partial}{\partial \dot t} + \left(\frac{3}{2}\frac{\ddot t ^2}{\dot t} + 2R(y,w) \dot t \right )\frac {\partial}{\partial \ddot t}\;.$$

The set of uniformizations $\mathscr Y \subset J^*(\mathbb{CP}_1,Y)$ is isomorphic to $J^*_2(\mathbb{CP}_1,Y)$ as an algebraic variety. We also have that the open subsets $J^*(Y,\mathbb{C})$ and $J^*(\mathbb{C},Y)$ are isomorphic as algebraic varieties. The ring of regular functions on $J^*(\mathbb{C},Y)$ is 
 $$
 \mathbb C[t][Y][y',\frac{1}{y'}, y'', \ldots]
 $$
and the isomorphism $J^*(Y,\mathbb{C})\simeq J^*(\mathbb{C},Y)$ is given by usual formula to express the derivation of a reciprocal function, namely $y' = \frac{1}{\dot t}$, $y'' = \frac{-\ddot t}{\dot t^3}$, ...
 
Notice that as pro-algebraic varieties $J^*(Y,\mathbb{CP}_1) \simeq J^*(\mathbb{CP}_1,Y) $ and under this isomorphism $\mathscr C$ and $\mathscr Y$ coincide. Moreover, the differential structures are not the same as on $J^*(\mathbb{CP}_1,Y)$ the differential structure of the structural ring is given by 

$$
\frac{d}{dt} =\frac{\partial}{\partial t} + y'\frac{\partial}{\partial y} + y''\frac{\partial}{\partial y'} +y'''\frac{\partial}{\partial y'} + \cdots.
$$
The subset $\mathscr Y$ is the zero set of the differential ideal generated by  $S_t(y) + 2R(y,w) y'^2=0$. As already mentioned $\mathscr Y$ is isomorphic to  
$J^*_2(Y,\mathbb{CP}_1)$. The $\frac{d}{dt}$ -differential structure induced  on its ring of regular maps $
 \mathbb C[t][Y][y', \frac{1}{y'},y''] $ is given by the above equation
 
 $$
 \frac{d}{dt} = \frac{\partial }{\partial t} + y'\frac {\partial}{\partial y} + y''\frac {\partial}{\partial y'} + \left(\frac{3}{2}\frac{{y'' }^2}{y'} - 2R(y,w) {y'}^3 \right )\frac {\partial}{\partial y''}\;.$$

To describe the connection form observe that the choice of our coordinates on open subset $J^\ast_2(Y,\mathbb C)$ induces a trivialization
$$Y\times \mathrm{PSL}_2(\mathbb C) \to J^\ast_2(Y,\mathbb{C P}_1),\quad 
\left(y,w, \begin{bmatrix}
a &  b \\ c & d 
\end{bmatrix} \right) \,\mapsto\, (y,w,t,\dot t, \ddot t)
$$
where $t = -b/a$, $\dot t = 1/a^2$, $\ddot t = -2c/a^ 3$, or equivalently $y' = a^2$, $y'' = 2ca^3.$ We see, by direct substitution, that the linear matrix differential equation in $Y\times \mathrm{PSL}_2(\mathbb C)$ is
\begin{equation}\label{eq:linearized}
\frac{dU}{dy} = A(y,w)U\quad\mbox{where}\quad   A(y,w) = \begin{bmatrix} 0 & 1 \\ R(y,w) & 0\end{bmatrix}
\end{equation}
and $R(y,w)$ is a rational function on $Y$ equivalent to
$$S_y(t)- 2R(y,w) = 0,\quad S_t(y) + 2R(y,w)y'^2=0$$ 
in the corresponding systems of coordinates on $J^\ast_2(Y,\mathbb C)$. Let us define the matrix-valued rational $1$-form on $J^\ast_2(Y,\mathbb{C P}_1)$:
$$\Omega = U^{-1}dU - U^{-1}A(y,w)Udy.$$
This $1$-form $\Omega$ is the \emph{connection form} of the differential equation. It has the following properties (some of which were already pointed out in Subsection \ref{CartanForm}):
\begin{enumerate}
\item[(a)] The kernel of $\Omega$ is the ${\rm PSL}_2(\mathbb C)$-connection $\mathcal F$ tangent to the graphs of solutions of the Schwarzian differential equation.
\item[(b)] $\Omega$ takes values in $\mathfrak{sl}_2(\mathbb C)$.
\item[(c)] If $X$ is the infinitesimal generator of a monoparametric group of right translations $\{R_{\exp(\varepsilon B)}\,\colon\,\varepsilon \in \mathbb C\}$
for certain $B\in\mathfrak{sl}_2(\mathbb C)$ then $\Omega(X) = B$.
\item[(d)] $\Omega$ is ${\rm adj}$-equivariant $R_g^*(\Omega) = {\rm Adj}_{g}^{-1}\circ \Omega$.
\item[(e)] $d\Omega + \frac{1}{2}[\Omega,\Omega] = 0$. 
\end{enumerate}

 \subsection{Algebraic relations between solutions of $n$ Schwarzian equations}
 
 Let us consider $Y_1,\ldots,Y_n$ affine algebraic curves, and for each $Y_i$ the bundle $J_i = J_2^*(\mathbb{CP}_1,Y_i)$. The product $\tilde J = J_1\times \ldots \times J_n$ is a $({\rm PSL}_2(\mathbb C))^n$-bundle over the product $\tilde Y = Y_1\times \ldots \times Y_n$. Let us consider $n$ Schwarzian equations,
 $$S_{t_i}y_i + 2R_i(y_i,w_i)y_i'^2 = 0.$$
 Each one is seen as a ${\rm PSL}_2(\mathbb C)$-invariant connection $\nabla_i$ in $J_i$ over $Y_i$ with connection form $\Omega_i$. We consider $\tilde{\nabla}$ the product $\tilde{\nabla} = \nabla_1 \times \ldots \times \nabla_n$ which is a $({\rm PSL}_2(\mathbb C))^n$ invariant connection on $\tilde J$ over $\tilde Y$.
 
 We are now ready to state the relevant version of Theorem \ref{Binary}. Notice that, since we are working on curves, the factors $Y_i$ have no positive dimensional special subvarieties.
 
\begin{thm}\label{MainSL2}
Let $(Y_i, \nabla_i)$ be algebraic curves with simple $({\rm PSL}_2(\mathbb C), \mathbb{CP}_1)$-structures. Assume that $$ \hat{\mathcal V}\colon {\rm Spf}\mathbb C[[s_1,\ldots,s_k]]\to \tilde J$$
is a component-wise non constant formal parameterized space in a horizontal leaf of $\tilde \nabla$, and let $V$ be the Zariski closure of $\hat{\mathcal V}$. The following are equivalent:
\begin{itemize}
    \item[(a)] $\dim V < 3n + {\rm rank}(\ker \Omega|_V)$.
    \item[(b)] There are two different indices $1\leq i<j \leq n$, a curve $X_{ij} \subset Y_i \times Y_j$ with both projections, $\pi_i$ and $\pi_j$ dominant such that
    \begin{itemize}
        \item $pr_{ij}(V) \subset X_{ij}$.
        \item there is a ${\rm PSL}_2(\mathbb C)$-bundle isomorphism between $\pi_i^*(J_i)$ and $\pi_j^*(J_j)$ on $X_{ij}$ such that $f_i^*\Omega_i = f_j^*\Omega_j$.
    \end{itemize}
\end{itemize} 
\end{thm}

In particular, under the hypothesis of Theorem \ref{MainSL2}, the coordinates $y_i$ and $y_j$ are algebraically dependent on $X_{ij}$. In other words it follows that if $\upsilon_i$ and $\upsilon_j$ are solutions of the corresponding equations then $\upsilon_i(t_i(s))$ and $\upsilon_j(t_j(s))$ are algebraically dependent over $\mathbb C$. 

Note that the rank of $\ker\Omega|_V$ is at least the dimension of the smallest analytic subvariety containing $\hat{\mathcal V}$ and thus is greater or equal to the rank of the Jacobian of our $n$ formal power series in $k$ variables. A more precise description of the possible dimension of a non-trivial intersection of $V$ with a leaf is:

\begin{cor}\label{dimensions}
Let $(Y_1, \nabla_1)$, $(Y_2, \nabla_2)$ be algebraic curves with simple $({\rm PSL}_2(\mathbb C), \mathbb{CP}_1)$-structures.
Consider an horizontal leaf $\mathcal L = \mathcal L_1 \times \mathcal L_2 \subset J_1 \times J_2$ and an algebraic subvariety $V \subset J_1 \times J_2$. If $\mathcal V$ is an irreducible component of  $V \cap \mathcal L$  and $\bar{\mathcal V}$ its Zariski closure, then $\dim \bar{\mathcal V}$ is $0$, $4$, $7$ or $8$. 
\end{cor}

\begin{proof}
Let us assume $V = \bar{\mathcal V}$. 

{ If the projection of $\mathcal V$ on a factor, say $Y_1$, is not dominant then its image is a point and its projection in $J_2$ is also a point. Then $V$ has dimension less than or equal to $4$. Its projection on the second factor gives a $\nabla_2$-invariant algebraic subset thus of dimension $0$ or $4$. }

Assume both projections are dominant. By lemma \ref{rank}, the rank of $\ker \Omega|_V$ can be $0$,$1$ or $2$. If it is $0$ then $\mathcal V$ is a point. From the above theorem, if $0<\dim V< 7$, then $V$ is the graph of a gauge correspondence between $\nabla_1$ and $\nabla_2$ thus $\dim V = \dim J_1 = \dim J_2 = 4$. The hypothesis that the structure is simple implies that the dimension cannot be smaller.
\end{proof}

\section{Proof of the Ax-Schanuel Theorem for products of curves.} \label{prodsofcurves}
In this section, we prove a general Ax-Schanuel type theorem for curves with geometric structure. Using this, we obtain two instances of the \emph{Ax-Schanuel Theorem with derivatives with a full characterization of algebraic relations}\footnote{For the remainder we will just say the \emph{full Ax-Schanuel.}}. Namely, we establish the case of products of hyperbolic curves (a direct generalization of the work \cite{AXS} and \cite{ASmpt} in the case of curves) and the case of non-hyperbolic curves given by generic triangle groups.

Let us carefully state the general problem for a Schwarzian equation. Let $\upsilon$ denote a solution of a Schwarzian equation
\begin{equation}\tag{$\star$} \label{stareqn}
S_t(y) + 2R(y,w)y'^2=0
\end{equation}
attached to a complex affine algebraic curve $Y$, and such that ${\rm Gal}(\nabla)={\rm PSL}_2(\mathbb{C})$. Notice that we take $\upsilon:U\rightarrow Y$ to be a uniformization function as defined in Subsection \ref{sectionUniformization}. {Let $\hat t_1,\ldots,\hat t_n$ be formal parameterizations of neighborhoods of points $p_1, \ldots, p_n$ in $U$, i.e., $\hat t_i \in p_i + (\mathfrak m \setminus \{0\})$ with $\mathfrak m$ the maximal ideal in $\mathbb C[[s_1, \ldots, s_\ell]]$. The Ax-Schanuel Theorem with derivatives with a full characterization of algebraic relations is an answer to the following problem.
\begin{prob} (Full Ax-Schanuel)
Fully characterize the conditions on $\hat t_1,\ldots,\hat t_n$ for which
$${\rm tr.deg.}_{\mathbb{C}}\mathbb{C}(\hat t_1,\upsilon(\hat t_1),\upsilon'(\hat t_1),\upsilon''(\hat t_1),\ldots,\hat t_n,\upsilon(\hat t_n),\upsilon'(\hat t_n),\upsilon''(\hat t_n))< 3n+{\rm rank}\left(\frac{\partial\hat t_i}{\partial s_j}\right).$$
\end{prob}}

Using Theorem \ref{MainSL2} (which holds since ${\rm Gal}(\nabla)={\rm PSL}_2(\mathbb{C})$) with $Y=Y_i$ and $\upsilon=\upsilon_i$ for $i=1,\ldots,n$, to prove the full Ax-Schanuel Theorem we only need to characterize the conditions on pairs $\hat t_1,\hat t_2$. Moreover, from Corollary \ref{dimensions}, if $\hat t_1,\hat t_2$ are nonconstant and 
$${\rm tr.deg.}_{\mathbb C} \left( \mathbb C (\hat t_1,\hat t_2, \upsilon(\hat t_1), \upsilon'(\hat t_1), \upsilon''(\hat t_1), \upsilon(\hat t_2), \upsilon'(\hat t_2), \upsilon''(\hat t_2)\right)=\ell < 7,$$ 
then it must be that $\ell=4$ and there is a polynomial $P \in \mathbb C[x,y]$ such that $P(\upsilon(\hat t_1), \upsilon(\hat t_2))=0.$ {In order to obtain the full Ax-Schanuel theorem, we will need to carefully study the way in which the polynomial $P$ is obtained in Theorem \ref{MainSL2}. To simplify notation, we write $t_i$ instead of $\hat t_i$. As one can expect, the relation $P(\upsilon(t_1), \upsilon(t_2))=0$ is not the only algebraic relations on the projective structure. The others are given by rational gauge transformations on the curve $X = \{P=0\} \subset Y \times Y$ transforming the connection  $\nabla_1 = {\rm d} - \begin{bmatrix}0& 1 \\ R(y_1) & 0 \end{bmatrix} {\rm d}y_1$ into the connection   $\nabla_2 = {\rm d} - \begin{bmatrix}0& 1 \\ R(y_2) & 0 \end{bmatrix} {\rm d}y_2$. Given a fundamental matrix $\Phi(y_1)$ for the first connection,
$$\Phi(y_1)  =\begin{bmatrix} \Phi_1(y_1)& \Phi_2(y_1)\\ \Phi_1'(y_1) & \Phi_2'(y_1) \end{bmatrix},$$ 
the gauge transformation is given by matrix of rational functions on $X$   
$$A(y_1, y_2) =\begin{bmatrix} a(y_1, y_2)& b(y_1,y_2) \\ c(y_1,y_2) & d(y_1,y_2) \end{bmatrix}$$
and a constant matrix $C = \begin{bmatrix} c_{11} & c_{12}\\c_{21} & c_{22} \end{bmatrix}\in\text{GL}_{2}(\mathbb{C})$ such that $$\Phi(y_2)C = A(y_1, y_2)\Phi(y_1),$$ where $\Phi(y_2)$ is a fundamental matrix for the second connection. 

Writing this transformation in the coordinates $y_1, y_2, t_1 = \frac{\Phi_1(y_1)}{\Phi_2(y_1)}, t_2 = \frac{\Phi_1(y_2)}{\Phi_2(y_2)}, y_1', \ldots, y_2''$, one obtains among all the algebraic relations, the following expression for $t_2$

$$
\frac{c_{11} t_2 + c_{21}}{c_{12} t_2 + c_{22}} = t_1 + \frac{\frac{b}{a}\frac{1}{y_2'}} {1+\frac{b}{2a}\frac{y_2''}{y_2'^2}}.
$$
In particular, if $b\equiv 0$ we have that $t_1=ht_2$ for an homography $h \in {\rm PSL}_2(\mathbb C)$. This in turn means that the two pull-backs of the $({\rm PSL}_2(\mathbb C), \mathbb{C P}_1)$ structure on $Y$ by the two projections of $X$ onto $Y$ will give the same projective structure on $X$. On the other hand if $b\not\equiv0$, then the projective structures on $X$ will be different.

\begin{defn}
A $({\rm PSL}_2(\mathbb C), \mathbb{C P}_1)$ structure on a curve $Y$, with uniformizing equation $(\star)$  is said to be {\em gauge invariant} if for any curve $X\subset Y\times Y$ and any rational gauge transformations on the curve $X$, the two pull-backs of the $({\rm PSL}_2(\mathbb C), \mathbb{C P}_1)$ structure on $Y$ by the two projections of $X$ onto $Y$ give the same projective structure on $X$.
\end{defn}

Given the above discussion, we have the following:}

\begin{prop}
\label{Ax-Schanuel-curve} Let $\upsilon$ be uniformization of a simple $({\rm PSL}_2(\mathbb C), \mathbb{C P}_1)$ structure on $Y$ which is gauge invariant. Let $\hat t_1$, $\hat t_2$ be two non-constant power series in $\mathbb C[[s]]$. If
$${\rm tr.deg.}_{\mathbb C} \left( \mathbb C ( \hat t_1,\hat t_2, \upsilon(\hat t_1), \upsilon'(\hat t_1), \upsilon''(\hat t_1), \upsilon(\hat t_2), \upsilon'( \hat t_2), \upsilon''(\hat t_2)\right)=4 $$
then there exist a polynomial $P\in \mathbb C[y_1, y_2]$ and an homography $h \in {\rm PSL}_2(\mathbb C)$
such that $\hat t_2 = h\hat t_1$ and $P(\upsilon(\hat t_1), \upsilon(\hat t_2))=0$.
\end{prop}

It is worth noting that another consequence of the conclusion of Proposition \ref{Ax-Schanuel-curve} is that if $\ell=4$, then the two functions $\upsilon(\hat t_1)$ and $\upsilon(\hat t_2)$ are solutions to the same {Schwarzian} equation with the derivation $\delta_1$. Hence, given Propositon \ref{Ax-Schanuel-curve}, the full Ax-Schanuel Theorem follows as soon as one is able to obtain a complete description of the structure of the set of solutions, in a differentially closed field, of the zero fiber equation. We will next look at two instances where we have been able to establish this.

\subsection{The case of hyperbolic curves} 

We assume in this subsection that $Y$ has negative Euler characteristic, $e(Y) <0$, and the projective structure has regular singularities. More precisely, let $\hat{Y}$ be a smooth projective completion of the affine curve $Y$. The $({\rm PSL}_2(\mathbb C), \mathbb{C P}_1)$ structure considered will be regular on $Y$ with regular singularities at  $p \in \hat Y \setminus Y$,meaning that  the singularity is given in local coordinate as $R(p+y) = \frac{1}{2} \frac{1-\mu^2}{y^2} + \frac{h(y)}{y}$, where $h(y)$ is holomorphic,  $\mu$ is the local exponent at the singular point. If $\mu =0$, the singularity is a cusp with a local solution $t(y) = \log(y) + h(y)$ where $h(y)$ is holomorphic. Furthermore, if $\mu \not \in \mathbb Z$, a local solution is of the form $t(y) = y^\mu h(y)$ where $h(y)$ is holomorphic. A regular point has local exponent $\mu= 1$. A singular point $p$ with local solution $t(y) = y^\mu h(y)$ with $h(y)$ holomorphic, and $\mu \in \mathbb N$, is a branching point of multiplicity $\mu-1$.

\begin{rem}
As the solutions of the Schwarzian equation, $S_y(t) = R(y)$, are quotients of solutions of the associated linear system, 
the exponent of the Schwarzian equation is the difference of the local exponents of the associated linear system.

\end{rem}

\begin{prop} \label{general curves}
Consider a regular singular projective structure on a curve $Y$ with $e(Y)<0$, such that for any two different non cuspidal singularities in $\hat Y \setminus Y$ with exponents $\mu_1$ and $\mu_2$, we have that $\mu_1$, $\mu_2$ and $1$ are linearly independent over $\mathbb Q$. Then the projective structure is gauge invariant.
\end{prop}

The main point in the following proof comes from \cite[Theorem 4.1]{BaldassarriDworkTovena}, which  proves that projective structures with small number of regular singularities are characterized by their monodromy representation up to conjugation. 
\begin{proof}
Let $X \subset Y\times Y$ be a curve and consider the two  projective structures on $X$ obtained by the pull-back of the structure on $Y$ by the first and by the second projection. Assume that there exists a rational gauge transformation between these two projective structures.

Choose $\hat X$ a projective normalization of $X$. {\it A priori} the coefficient of the gauge transformation $b^2$ on $\hat X$, is defined on an open subset where the coordinates $y_1$ and $y_2$ are \'etale. Following \cite{BaldassarriDworkTovena}, these $b^2$'s glue together in a rational section of $T \hat X ^{\otimes 2}$ (still denoted by $b^2$) and then $\sum_{x\in \hat X} ord_x(b^2) = 2 e(\hat X)$ or $b$ is zero. Here $\chi$ is the Euler characteristic. Then \cite[Theorem 4.1]{BaldassarriDworkTovena} bounds the vanishing order of $b^2$ by the singularities data. The argument of the proof of the proposition is as follows.

Both projections of $\hat X$ on $\hat Y$ are ramified coverings. Consider the two projective structures on $\hat X$ obtained by pull-back of the projective structure of $\hat Y$ by the two projections. For such a projective structure, the ramification points $p \in \hat X$ of the projection  will be singular points with exponent $n\mu$ where $n$ is the ramification order and $\mu$ is the exponent of the projection of $p$ in $\hat Y$.


As the monodromy around singularities of both projective structures is the same, their exponents at a singular point must differ by an integer. As two non-cuspidal exponents at singular points in $\hat Y$ together with $1$ are $\mathbb Q$ linearly independent, the exponents of the two projective structures at the same point $p \in \hat X$ must be equal and its projections on $\hat Y$ by the two projections coincide.

Let $b_1$ and $b_2$ be the number of branch points (with multiplicities) of the respective projective structures, $\ell_1$ and $\ell_2$ be the number of nonapparent singularities, and $m_1$, $m_2$ the degrees of the projections. By the Riemann-Hurwitz formula $b_i  = \ell_i + m_i e(Y) - e(\hat X)$. In particular $\ell_1 + \ell_2 -b_1 -b_2 = 2 e(\hat X) - (m_1+m_2)e(Y) > 2 e(\hat X)$. 

Theorem 4.1 from \cite{BaldassarriDworkTovena} applied to this particular case, gives that if the projective structures are different and have the same monodromy,  then $\ell_1 + \ell_2 -b_1 -b_2 \leq 2 e(\hat X)$. So the two projective structures must be the same, that is, the projective structure is gauge invariant.
\end{proof}

Let $\Gamma\subset {\rm PSL}_2(\mathbb{R})$ be a Fuchsian group of the first kind and $Y = \mathbb{H}/\Gamma$. If $C_{\Gamma}$ denotes the set of cusps of $\Gamma$ and $\mathbb{H}_{\Gamma}:=\mathbb{H}\cup C_{\Gamma}$, then there is a meromorphic mapping $j_{\Gamma}:\mathbb{H}_{\Gamma}\rightarrow \hat{Y}(\mathbb{C})$. The map $j_{\Gamma}$ is called a {\it uniformizer} and is an automorphic function for $\Gamma$
\[j_{\Gamma}(g t)=j_{\Gamma}(t)\;\;\;\text{ for all }\;g\in\Gamma\text{ and } t \in\mathbb{H}\]
and so factorizes in a bi-rational isomorphism of $\Gamma\backslash \mathbb{H}_{\Gamma}$ into $\hat{Y}(\mathbb{C})$. We have that $j_{\Gamma}$ is a solution of the Schwarzian equation $(\star)$ for some rational function $R$ on $Y$ with regular singularities. Note that in this case we have $e(Y)<0$ and then {\it a fortiori} the Euler characteristic of the regular locus is negative.

\begin{thm}[\bf Theorem D]\label{AxSchanuelHyperbolic}
Let $\hat t_1,\ldots,\hat t_n \in \mathbb C[[s]]\setminus \mathbb C$ be formal parameterizations of neighborhoods of points $p_1,\ldots, p_n$ in $\m H$. Assume that $\hat t_1,\ldots,\hat t_n$ are geodesically independent, namely $\hat t_i$ is nonconstant for $i =1, \ldots , n$ and there are no relations of the form $\hat t_i = \gamma\hat t_j$ for $i \neq j$, $i,j \in \{1, \ldots , n \}$ and $\gamma$ is an element of  $\text{Comm}(\Gamma)$, the commensurator of $\Gamma$. Then
$${\rm tr.deg.}_{\mathbb{C}}\mathbb{C}(\hat t_1,j_{\Gamma}(\hat t_1),j'_{\Gamma}(\hat t_1),j''_{\Gamma}(\hat t_1),\ldots,\hat t_n,j_{\Gamma}(\hat t_n),j'_{\Gamma}(\hat t_n),j''_{\Gamma}(\hat t_n))\geq 3n+{\rm rank}\left(\frac{\partial \hat t_j}{\partial s_i}\right).$$
\end{thm}

\begin{proof}

For notational simplicity we write $t_i$ instead of $\hat t_i$. Assume that the inequality does not hold and that none of the $t_i$'s are constants. Then as discussed above, using Theorem \ref{MainSL2} and Corollary \ref{dimensions} we have a couple of indices, say 1 and 2, such that 
$${\rm tr.deg.}_{\mathbb C} \left( \mathbb C ( t_1, t_2, j_\Gamma(t_1), j_\Gamma'(t_1), j_\Gamma''(t_1), j_\Gamma(t_2), j_\Gamma'(t_2), j_\Gamma''(t_2)\right)=4.$$

If the projective structure is simple and gauge invariant, then we can apply Theorem \ref{Ax-Schanuel-curve} to get that $t_2 = ht_1$ for some homography $h \in {\rm PSL}_2(\mathbb C)$. So the two functions $j_\Gamma(t_1)$ and $j_\Gamma(t_2)$ are solutions to the same zero fiber equation ($\star$). Using Theorem \ref{Urank1}, we have that the set defined by this equation, in a differentially closed field, is strongly minimal and satisfies the refined version of geometric triviality. We also have that the weak form of the Ax-Lindemann-Weierstrass Theorem given in Corollary \ref{independence} holds. In particular $h\in\text{Comm}(\Gamma)$ and we are done.

Let us verify that the projective structures given by Fuchsian uniformizations are simple and gauge invariant. Using \cite[Corollary B.1]{Morales}, we have that $\Gamma$ is Zariski dense in ${\rm Gal}(\nabla)$, for the corresponding connection $\nabla$. So ${\rm Gal}(\nabla)=G={\rm PSL}_2(\mathbb{C})$ and the structure is simple. 

If $\Gamma$ has no torsion, then all the singularities of the projective structure are cusps. So Proposition \ref{general curves} applies and we are done.

Now if $\Gamma$ has a finite order element, by Selberg's Lemma there exists a finite index subgroup $\Gamma_0< \Gamma$ such that $\Gamma_0$ has no torsion. Let $Y_0$ be the quotient $\Gamma_0\backslash \mathbb{H}$. The curve $\hat Y_0$ is a ramified covering of $\hat Y$. If $X_0$ is an irreducible component of the pull-back of $X$ in $Y_0 \times Y_0$, then Proposition \ref{general curves} can be applied and the two pull-back of the projective structure of $\hat Y$ on $\hat X_0$ coincide. This implies that the intermediate pullbacks on $\hat X$ coincide and so the projective structures are the same.
\end{proof}

We hence also answer positively a question of Aslanyan \cite[Section 3.4]{Aslanyan} about the existence of differential equations satisfying the full Ax-Schanuel Theorem and such that the polynomial ($X_1-X_2$) is the only $\Gamma$-special polynomial. Indeed, this is true of any $\Gamma$ satisfying $\Gamma=\text{Comm}(\Gamma)$. Many such examples exist (cf. \cite[Fact 4.9]{BCFN}). We next give some more examples.

\subsection{The case of generic triangle}

Let $\triangle\subset\m H$ be an open circular triangle with vertices $v_{1},v_{2},v_{3}$ and with respective internal angles $\frac{\pi}{\alpha}$, $\frac{\pi}{\beta}$ and $\frac{\pi}{\gamma}$. Using the Riemann mapping theorem, we have a unique biholomorphic mapping $J:\triangle\rightarrow \mathbb{H}$ sending the vertices $v_{1},v_{2},v_{3}$ to $\infty,0,1$ respectively. We can also extend $J(t)$ to a homeomorphism from the closure of $\triangle$ onto $\overline{\mathbb{H}}=\mathbb{H}\cup{\bf P}^1(\mathbb{R})$. The function $J(t)$ is called a {\em Schwarz triangle function} and is a uniformization function in the sense of Subsection \ref{sectionUniformization}. It satisfies the Schwarzian equation ($\star$) with
\[R(y,w)=R_{\alpha,\beta,\gamma}(y)=\frac{1}{2}\left(\frac{1-{\beta}^{-2}}{y^2}+\frac{1-{\gamma}^{-2}}{(y-1)^2}+\frac{{\beta}^{-2}+{\gamma}^{-2}-{\alpha}^{-2}-1}{y(y-1)}\right).\]
We call such an equation a Schwarzian triangle equation and write it as $\chi _{\alpha,\beta,\gamma, \frac{d}{dt}} (y)=0.$ By a generic such equation we mean the case when $\alpha,\beta,\gamma\in\mathbb{R}_{>1}$ are algebraically independent over $\mathbb{Q}$.

Let us now state what is known about generic Schwarzian triangle equations. The proof can be found in \cite[Section 4]{BCFN}. 
\begin{fct}\label{BCFN-main} Let $\alpha,\beta,\gamma\in\mathbb{R}$ be algebraically independent over $\mathbb{Q}$. Then
\begin{enumerate}
    \item The Galois group ${\rm Gal}(\nabla)$ for the corresponding connection $\nabla$ is ${\rm PSL}_2(\mathbb{C})$.
    \item The set defined by $\chi _{\alpha,\beta,\gamma, \frac{d}{dt}} (y)=0$ is strongly minimal, geometrically trivial and strictly disintegrated. Namely, if $K$ is any differential field extension of $\mathbb{Q}(\alpha,\beta,\gamma)$ and $y_1,\ldots,y_n$ are distinct solutions that are not algebraic over $K$, then $${\rm tr.deg._KK}(y_1,y'_1,y''_1,\ldots,y_n,y'_n,y''_n)=3n.$$
\end{enumerate}
\end{fct}
We now assume that $\alpha,\beta,\gamma\in\mathbb{R}_{>1}$ are 
algebraically
independent 
over $\mathbb{Q}$ and let $J:\triangle\rightarrow {\mathbb{H}}$ be the corresponding Schwarz triangle function as describe above. We obtain the full Ax-Schanuel Theorem:

\begin{thm}\label{AxSchanuelGenericTriangle}
Let $\hat t_1,\ldots,\hat t_n$ be distinct formal parameterizations of neighborhoods of points $p_1,\ldots, p_n$ in $\triangle$. Then
$${\rm tr.deg.}_{\mathbb{C}}\mathbb{C}(\hat t_1,J(\hat t_1),J'(\hat t_1),J''(\hat t_1),\ldots,\hat t_n,J(\hat t_n),J'(\hat t_n),J''(\hat t_n))\geq 3n+{\rm rank}\left(\frac{\partial \hat t_j}{\partial s_i}\right).$$
\end{thm}

\begin{proof} As with the proof of Theorem \ref{AxSchanuelHyperbolic}, all we have to do is prove that the projective structure is simple and gauge invariant. Simplicity follows from Fact \ref{BCFN-main}(1). Since $1, \alpha^{-1},\beta^{-1}, \gamma^{-1}$ are linearly independent over $\mathbb Q$ and $e(\mathbb{CP}_1 \setminus \{0,1, \infty \}) < 0$, Proposition \ref{general curves} gives gauge invariance. The problem is now reduced to the weak form of the Ax-Lindemann-Weierstrass Theorem which was established in Fact \ref{BCFN-main}(3).

\end{proof}

\begin{rem}
In \cite{BCFN}, the parameters $\alpha,\beta,\gamma$ where allowed to be arbitrary complex numbers. The reader can follow the proofs above and see that the proof are correct in this less geometrical context. However the function $J$ is no more a biholomorphism of a triangle onto the upper half plane but a solution of a Schwarzian equation defined on a simply connected open subset of $\mathbb{CP}_1\setminus\{0,1,\infty\} $. 
\end{rem}

\section{Ax-Schanuel for combined uniformizers}
\label{combinedexpj}
Our theorem may be applied to geometric structures of different nature on curves. For instance the exponential function is the uniformization of a $(\mathbb G_a(\mathbb C), \mathbb A(\mathbb C))$ structure on the curve $Y = \mathbb C\setminus\{0\}$ whose charts are determinations of the logarithm up to an additive constant. The groups are $G =\mathbb G_a(\mathbb C)$ and $B=\{e\}$ so that $G/B = \mathbb A(\mathbb C)$. Even if the additive group is not simple, the group ${\rm \mathbb G_a}(\mathbb C)^k \times {\rm PSL}_2(\mathbb C)^\ell$ is sparse so we can apply Theorem \ref{th:AxSh1} in this combined case.

For $i=1, \ldots, k$ consider a rational function $R_i$ on a curve $\hat Y_i$ and $Y_i$ the subset defined by $R_i\not = 0$ seen as a ramified covering of $\mathbb C$ with coordinate $z$. For any $i$, let $U_i \subset \mathbb{CP}_1$ be a Euclidean open subset with affine coordinate $t$ such that there exist an holomorphic $\hat e_i \in \mathcal O(U_i)$ solution of $ \frac{dz}{dt} = R_i(z)$ transcendental over $\mathbb C(t)$.   These are simple $(\mathbb G_a(\mathbb C), \mathbb A(\mathbb C))$-structures on $Y_i$.

Consider a simple gauge invariant $({\rm PSL}_2(\mathbb C) , \mathbb{CP}_1)$-structure on a curve $Y$ and let $\upsilon \in \mathcal O(U)$ be a uniformizer of this structure.

\begin{thm}\label{mixedThm}
Let $\hat t_1, \ldots \hat t_{k+\ell}$ be non constant formal power series in $\mathbb C [[s_1, \ldots s_m]]$ with $ \hat t_i(0) \in U_i$ for $i=1, \ldots, n$ and $\hat t_{k+1}(0), \ldots \hat t_{k+\ell}(0) \in U$.
If 
$${\rm tr.deg.}_{\mathbb C}\mathbb C \left(\hat t_1, \ldots \hat t_{k+\ell}, \hat e_1(\hat t_1), \ldots \hat e_k(\hat t_k), \upsilon (\hat t_{k+1}), \ldots \upsilon ''(\hat t_{k+\ell})\right) <  k + 3\ell + {\rm rank}\left( \frac{\partial \hat t_i}{\partial s_j}\right)$$ 
then 
\begin{itemize}
\item there exists $n\in \mathbb C^k\setminus\{0\}$ such that $\sum_{i=1}^k n_i \hat t_i \in \mathbb C(\hat e(\hat t_1), \ldots \hat e(\hat t_k))$ and ${\rm tr.deg.}_{\mathbb C}\mathbb C(\hat e(\hat t_1), \ldots \hat e(\hat t_k)) < k$; or 
\item 
there exist $k+1\leq i < j\leq k+\ell$ and $\gamma \in {\rm PSL}_2(\mathbb C)$ such that $\hat t_j = \gamma(\hat t_i)$ and ${\rm tr.deg.}_{\mathbb C} \mathbb C(\upsilon(\hat t_j), \upsilon(\hat t_i)) = 1$.

\end{itemize}
\end{thm}

\begin{proof}
Let $t_i$ be the coordinate on $\mathbb C$ and $e_i$ the coordinate on $Y_i$. The connection $\nabla_i$ on $\mathbb C \times Y_i \to Y_i$ given by $dt_i - \frac{de_i}{R_i(e_i)} = 0$ is a $(\mathbb C,+)$-principal connection.

Let $\nabla$ be the connection on ${\rm PSL}_2(\mathbb C) \times Y \to  Y$ associated to the Schwarzian equation, described in Section \ref{CurveCase}. It is a ${\rm PSL}_2(\mathbb C)$-principal connection.

We apply Theorem \ref{th:AxSh1} to the principal connection $(\prod \nabla_i) \times (\nabla)^{\ell}$ on $\left (\mathbb C ^k \times {\rm PSL}_2(\mathbb C) ^\ell \right) \times (\prod {Y_i}) \times {Y_2}^\ell \to  (\prod{Y_i}) \times {Y_2}^\ell$ defined for $v_1+v_2+\ldots v_{k+\ell} \in T_p((\prod{Y_i}) \times {Y_2}^\ell)$ by 
$$
\nabla_{v_1+v_2+\ldots v_{k+\ell}} = \sum_{i=1}^k (\nabla_i)_{v_i} + \sum_{i=k+1}^{k+\ell} (\nabla)_{v_i}.
$$ 

As the subset parameterized by 
$$
\left\{ \begin{array}{lcl} t_1 =\hat t_1(s), & \ldots ,& t_{k+l} = \hat t_{k+l}(s), \\

\dot t_{k+1} = \frac{1}{\upsilon'(\hat t_{k+1}(s))} ,& \ldots,& \ddot t_{k+l} = \frac{1}{\upsilon'(\hat t_{k+l}(s))}, \\

e_1= \hat e_1(\hat t_1(s)),& \ldots,& e_k = \hat e_k(\hat t_k(s)),\\ 

y_1 = \upsilon(\hat t_{k+1}(s)),&  \ldots, &y_{l} = \upsilon(\hat t_{k+l}(s))
\end{array}
\right.
$$

\noindent lies in a horizontal leaf of the connection and has a small enough Zariski closure $V$, we can use Theorem \ref{th:AxSh1} and get that the projection of $V$ in $(\prod {Y_i})\times {Y_2}^\ell$ is contained in a proper subvariety $Z$ (giving polynomial relation involving only the $e$'s and the $\upsilon$'s).

The Galois group of the restriction of the connection above a Zariski open subset $Z^\circ$ of $Z$ is a strict algebraic subgroup  $G \subset \mathbb G_a(\mathbb C)^k \times {\rm PSL}_2(\mathbb C)^\ell$ whose projections on each factors are onto by Galois correspondence.

By Goursat's lemma, the projection of $G$ in $\mathbb G_a(\mathbb C)^k $ or the projection in ${\rm PSL}_2(\mathbb C)^\ell$ is not onto. Applying Lemma \ref{lm:projection} we get that:
\begin{itemize}
    \item either the projection of $Z$ in $\prod Y_i$ is contained in a special subvariety, meaning that 
    
    \begin{itemize}
        \item it is a proper subvariety: ${\rm tr.deg.}_{\mathbb C}\mathbb C(\hat e(\hat t_1), \ldots \hat e(\hat t_k)) < k$,
\item  the Galois group of the restriction is given by a proper subvector space of $\mathbb C^k$: there exists $n\in \mathbb C^k\setminus\{0\}$ such that $\sum_{i=1}^k n_i \hat t_i \in \mathbb C(\hat e(\hat t_1), \ldots \hat e(\hat t_k))$;
       \end{itemize} 
    \item or the projection of $Z$ in $Y^\ell$ is contained in a special subvariety. Theorem \ref{Ax-Schanuel-curve} can be used to obtain the conclusion.
\end{itemize} 
\end{proof}

Using the original proof by Ax of the classical Ax-Schanuel Theorem, Theorem \ref{mixedThm} can be made more precise in the case of the exponential function and the uniformizer of a hyperbolic curve.

Let $Y$ by an hyperbolic curve  and denoted by  $j_{\Gamma}:\mathbb{H}_{\Gamma}\rightarrow \hat{Y}(\mathbb{C})$ its Fuchsian uniformization.

\begin{thm}[\bf Theorem E]\label{combined}
Let $\hat t_1,\ldots,\hat t_{k+\ell}$ be formal parameterizations (in variables $s=(s_1, \ldots, s_m)$) of neighborhoods of points in $\m H$. If 
$${\rm tr.deg.}_{\mathbb C}\mathbb C \left(\hat t_1, \ldots \hat t_{k+\ell}, \exp(\hat t_1), \ldots\exp(\hat t_{k}),  j_{\Gamma} (\hat t_{k+1}), \ldots j_{\Gamma} ''(\hat t_{k+\ell})\right) < k + 3\ell+{\rm rank}\left(\frac{\partial \hat t_j}{\partial s_i}\right)$$ 
then 
\begin{itemize}
\item there exists $n\in \mathbb Z^{k}\setminus \{0\}$ such that $\sum_{i=1}^k n_i \hat t_i \in \mathbb C$; or
\item there exist $k+1\leq i < j\leq k+\ell$ and $\gamma \in \text{Comm}(\Gamma)$ such that $\hat t_i = \gamma\hat t_j$.
\end{itemize}
\end{thm}

\begin{proof}
Applying Theorem \ref{combined} to the case where $\hat e_i=\exp$ for all $i=1,\ldots,k$ and $\upsilon=j_{\Gamma}$, we get that 
\begin{itemize}
\item there exists $n\in \mathbb C^{k}\setminus \{0\}$ such that 
$$\sum_{i=1}^{k} n_i \hat t_i \in \mathbb C\left(\exp(\hat t_1), \ldots \exp(\hat t_{k}))\right)\text{; or}$$ 
\item there exist $k+1\leq i < j\leq k+\ell$, $\gamma \in {\rm PGL}_2(\mathbb C)$ such that $\hat t_i = \gamma(\hat t_j)$ and  ${\rm tr.deg.}_{\mathbb C} \mathbb C(j_{\Gamma}(\hat t_j), j_{\Gamma}(\hat t_i)) = 1$. 
\end{itemize}

In the first case, we can use the original proof by Ax \cite{Ax71} of Ax-Schanuel theorem to conclude in the following way. On the Zariski closure $V \subset \mathbb C^k$ of the formally parameterized subspace $\exp(\hat t_1), \ldots \exp(\hat t_{k})$, we get $k$ rational functions $e_1, \ldots e_k$ and closed linear combination of logarithmic differential forms: $\sum_{i=1}^{k} n_i \frac{de_i}{e_i} = df$ with $f  \in \mathbb C(V)$. Applying \cite[Proposition 2]{Ax71}, we can assume $\sum_{i=1}^{k} n_i \frac{de_i}{e_i} = 0$ with $n \in \mathbb Z^n$. Then $\sum_{i=1}^{k} n_i t_i$ is constant on $V$.

In the second case, the $(i, j)$-projection of $P$ in ${\rm PSL}_2(\mathbb C)^2 \times Y^2$ gives a principal bundle on a curve in $Y^2$ for the subgroup $\{(h,\gamma h \gamma^-1): h \in {\rm PSL}_2(\mathbb C) \}$. Using Theorem \ref{AxSchanuelHyperbolic} it follows that for some $\gamma \in \text{Comm}(\Gamma)$  we have that $\hat t_i = \gamma\hat t_j$ as required.
\end{proof}

\end{document}